\documentclass[hidelinks,onefignum,onetabnum]{siamart220329}

\usepackage{amsmath} 
\usepackage{amssymb} 
\usepackage{amsfonts}
\usepackage[utf8]{inputenc}
\usepackage{graphicx}
\usepackage{float} 
\usepackage{subcaption} 
\usepackage{epstopdf}
\usepackage{multirow}
\usepackage{algorithmic}
\usepackage{tabularx} 
\usepackage{longtable} 
\ifpdf
  \DeclareGraphicsExtensions{.eps,.pdf,.png,.jpg}
\else
  \DeclareGraphicsExtensions{.eps}
\fi

\newsiamremark{remark}{Remark}
\newsiamremark{hypothesis}{Hypothesis}
\crefname{hypothesis}{Hypothesis}{Hypotheses}
\newsiamthm{claim}{Claim}

\headers{A refined joint bidiagonalization method}{K. FANG and Z. JIA}

\title{The Refined Joint Bidiagonalization Method and an Implicitly Restarted Algorithm for Large GSVD Computations\thanks{Submitted to the editors DATE.
\funding{This work was supported in part by the National Science Foundation of China (No. 12571404).}}}

\author{Kaixiao Fang\thanks{Department of Mathematical Sciences, Tsinghua University, 100084 Beijing, People's Republic of China.
  (\email{13021088186@163.com}).}
\and Zhongxiao Jia\thanks{Correspondence. Department of Mathematical Sciences, Tsinghua University, 100084 Beijing, People's Republic of China.
  (\email{jiazx@tsinghua.edu.cn}).}}

\usepackage{amsopn}

\begin{document}
\sloppy 

\maketitle

\begin{abstract}
    We make a convergence analysis on the joint bidiagonalization (JBD) method that
    computes several extreme generalized singular value decomposition (GSVD)
    components of a regular matrix pair $\{A,L\}$, and show that the right and
    left Ritz vectors obtained by
    it may converge erratically and even may fail to converge, while Ritz values converge.
    These convergence results
    hold for a class of general Rayleigh--Ritz projection methods for the GSVD problem
    under the hypothesis that the deviation of a
    desired right generalized singular vector from the right subspace tends to zero. We prove
    the interlacing property of Ritz values and generalized singular values, and
    extend it to the generalized singular values of $\{A,L\}$ and the matrix pairs consisting of subsets
    of its columns. To overcome the
    irregular convergence or possible non-convergence of the JBD method,
    we nontrivially extend the refined Rayleigh--Ritz projection for the eigenvalue problem to the GSVD
    problem, and propose a refined JBD (RJBD) method that replaces
    the right Ritz vectors by new approximations, called the right refined
    Ritz vectors, satisfying certain residual optimality;
    we define new approximate left generalized singular vectors, called the left refined Ritz vectors. We
    prove that the left and right refined Ritz vectors unconditionally converge under the same hypothesis.
    We extend the implicit restarting scheme to the RJBD method, and develop an implicitly restarted RJBD
 algorithm with the refined shifts proposed. Numerical
    experiments illustrate that the new algorithm is at least competitive and often considerably
    more efficient than the implicitly restarted JBD algorithm.
\end{abstract}



\section{Introduction}
Let $A \in \mathbb{R}^{m \times n}$ and $L \in \mathbb{R}^{p \times n}$, and suppose the matrix pair $\{A, L\}$ is regular, i.e., $\operatorname{rank}((A^{\mathrm T},\,L^{\mathrm T})^{\mathrm T})=n$, where the superscript T denotes the transpose of a matrix or vector. Denote $q_1=\operatorname{dim}(\mathcal{N}(A))$, $q_2=\operatorname{dim}(\mathcal{N}(L))$, $q=n-q_1-q_2$, $l_1=\operatorname{dim}(\mathcal{N}(A^{\mathrm T}))$,  and $l_2=\operatorname{dim}(\mathcal{N}(L^{\mathrm T}))$, where $\mathcal{N}(\cdot)$ is the null space of a matrix. Then according to \cite{PaigeandSaunders1981}, the GSVD of $\{A, L\}$ is
\begin{equation}
\label{GSVD in matrix terms}
\left\{\begin{array}{l}
P_A^{\mathrm T} A X=C_A=\operatorname{diag}\left\{C, 0_{l_1, q_1}, I_{q_2}\right\}, \\
P_L^{\mathrm T} L X=S_L=\operatorname{diag}\left\{S, I_{q_1}, 0_{l_2, q_2}\right\},
\end{array}\right.
\end{equation}
where $P_A$ and $P_L$ are orthogonal, $X$ is nonsingular, $C_A=\operatorname{diag}(c_1, c_2, \ldots, c_q)$ and
$S_L=\operatorname{diag}(s_1, s_2, \ldots, s_q)$ with $0<c_i, s_i<1$, and $c_i^2+s_i^2=1$, and the subscripts
in $0$ and $I$ represent the column and row sizes of zero and identity matrices; see
\cite{GolubandVan2013matrix} for more details. Throughout the paper, we label the $c_i$ in
descending order.

Partition $P_A=(P_{A,q},P_{A,l_1},P_{A,q_2})$, $P_L=(P_{L,q},P_{L,q_1},P_{L,l_2})$, and
$X=(X_q,X_{q_1},X_{q_2})$,
where the lowercase letters in subscripts represent the column numbers of matrices, and write
$X_q=(x_1,\ldots,x_q)$, $P_{A,q}=(p_1^A, \ldots, p_q^A)$, and $P_{L,q}=(p_1^L, \ldots, p_q^L)$. We call the
quintuple $\{c_i,\,s_i,\,x_i,\,p_i^A,\,p_i^L\}$ a nontrival GSVD component of $\{A,L\}$ with the
generalized singular value $c_i/s_i$ or $\{c_i,s_i\}$,
the left generalized singular vectors $p_i^A$, $p_i^L$ and the right
generalized singular vector $x_i$. The GSVD components satisfy the relationships
\begin{equation}
\label{GSVD in vector terms}
    \left\{\begin{array}{rl}
    A x_i & =c_i p_i^A, \\
    L x_i & =s_i p_i^L, \\
    s_i A^{\mathrm{T}} p_i^A & =c_i L^{\mathrm T} p_i^L,
    \end{array} \quad i=1,\,2,\dots, q.\right.
\end{equation}
$(0_{l_1,q_1},I_{q_1},P_{A,l_1},P_{L,q_1},X_{q_1})$ and $(I_{q_2},0_{l_2,q_2},P_{A,q_2},P_{L,l_2},X_{q_2})$ are called the trivial GSVD parts,
which correspond to the $q_1$ zero and $q_2$ infinite generalized singular values, respectively. A trivial
GSVD component satisfies the above relationships too by defining the pair $\{c,s\}=\{0,1\}$ or
$\{1,0\}$.

In this paper we aim to compute several GSVD components corresponding to the nontrivial extreme, i.e.,
largest or smallest generalized singular values. Such kind of problem arises from numerous applications,
such as solutions of linear discrete ill-posed problems with general-form
regularization
\cite{Rank-Deficient-and-Discrete-Ill-Posed-Problems}, principal component analysis in statistics \cite{PCA},
pattern recognition \cite{model-identification}, and signal processing \cite{Signal-Analysis}. The JBD method proposed by Zha \cite{Zha-JBD} can be used to solve this kind of problem, in which
the JBD process jointly reduces $\{A,\,L\}$ to upper bidiagonal forms.
Kilmer, Hansen and Español \cite{Kilmer-JBD} adapt the
JBD process to jointly bidiagonalizing $\{A,\,L\}$ to lower and upper bidiagonal forms, and exploit
it to solve linear discrete ill-posed problems with general-form Tikhonov regularization.
Jia and Yang \cite{JIAandYang2020JBDforTikhonov}
use it to solve large-scale linear discrete ill-posed problems with the general-form
regularization
and propose an equally effective but more efficient purely iterative solver without explicit Tikhonov regularization. The framework of general Rayleigh--Ritz projection methods is proposed in \cite{JiaandZhang2023FEASTGSVD}
for the large GSVD computation, which includes the JBD method, the Jacobi--Davidson (JD) type GSVD
method \cite{Huang-Numerical-experiments} and the CJ-FEAST GSVDsolver \cite{JiaandZhang2023FEASTGSVD}.

In this paper, we will make three main contributions to the GSVD problem and its numerical
solution. {\em As the first contribution}, we prove that the approximate left and right
generalized singular vectors, called the left and right Ritz
vectors, obtained by the JBD method may converge erratically and even may fail to converge
under the necessary hypothesis that the deviations of the desired right generalized singular vectors from
the right projection subspaces converge to zero.
These convergence results hold for a subclass of general Rayleigh--Ritz projection methods
\cite{JiaandZhang2023FEASTGSVD}, which will be clear later, under the hypothesis that
the deviations of
the desired generalized right singular vectors from right subspaces tend to zero.
Importantly and interestingly, we establish the interlacing property of
the generalized singular values and Ritz values, which becomes the strict one when the JBD method is used. The
interlacing property extends to the generalized singular values of the matrix pair $\{A,L\}$ and submatrix
pairs consisting of its arbitrary columns.

The refined Rayleigh--Ritz projection method initially proposed by Jia \cite{Jia-refined}
has been widely used \cite{book:BaiEigenvalue,book:StewartMatrixAlgorithms,book:VanDerVorst2002}, and has
been adapted to the large SVD computations in, e.g.,
\cite{Hochstenbach2004BIT,HuangandJia2019JD,Jia-IRRBL,implicit-restart-bidiag}.
The method computes new approximate eigenvectors or singular vectors,
called the refined Ritz vectors, which are guaranteed to unconditionally converge
and overcome the irregular convergence and possible non-convergence of Ritz vectors
under the natural hypothesis that the deviations of desired eigenvectors or singular vectors
from projection subspaces approach zero.
{\em As the second contribution of this paper}, we nontrivially extend the refined Rayleigh--Ritz projection to
the large GSVD computation, and propose a
refined JBD (RJBD) method that replaces the right Ritz vectors by new approximate right singular vectors,
called the right refined Ritz vectors, that satisfy certain residual optimality; exploiting them,
we introduce new approximate left generalized singular vectors, called left refined
Ritz vectors. We prove that the left and right refined Ritz vectors and the Ritz values are guaranteed to
converge when the mentioned deviations tend to zero.

The authors \cite{IRJBD} adapt the implicit restart technique proposed by Sorensen
\cite{implicit-restarted-Arnoldi} to the JBD process, and have developed an implicitly restarted JBD
(IRJBD) algorithm with several subtle issues addressed in finite precision arithmetic.
Numerical experiments in \cite{IRJBD} have demonstrated that IRJBD is at
least competitive with the thick-restart JBD algorithm \cite{TRJBD} in terms of restarts
and robustness. It turns out that one of the keys for the success and overall efficiency
of IRJBD  is proper selection of shifts involved.
The exact shifts have been proposed for IRJBD in \cite{IRJBD}, which are those
unwanted Ritz values and have been shown to be the best possible ones within
the framework of the JBD method. {\em As the third contribution of this paper},
we shall apply the implicit restarting scheme to the RJBD method,
and develop an implicitly
restarted RJBD algorithm, called IRRJBD. For the algorithm to be
efficient as much as possible, we analyze the exact shifts and get
more insight into them; motivated by it, we exploit the left and right refined Ritz vectors to present new
shifts, called the refined shifts, and
show that they can generate more accurate restarted subspaces than the exact shifts
do. Numerical
experiments will illustrate that IRRJBD is at least as fast as and is, in two-thirds of
the cases, $30\%\sim 50\%$ more efficient than IRJBD.

In \cref{sec:2}, we review the JBD method and IRJBD algorithm.
In \cref{sec:3}, we analyze the convergence of the Rayleigh--Ritz method including the JBD method, and show
why the Ritz vectors may converge erratically and even may
fail to converge. In \cref{sec:4}, we propose the RJBD method, and establish
the unconditional convergence of refined Ritz vectors.
\Cref{sec:5} is devoted to the refined shifts,
and we prove that they are better than the exact
shifts. \Cref{sec:6} summarizes IRRJBD with its default parameters.
In \cref{sec:7} we report numerical experiments to illustrate the performance of IRRJBD and its
superiority to IRJBD. \Cref{sec:8} concludes the paper with future work.

Throughout the paper, we denote by $\|\cdot\|$ the 2-norm of a matrix, by
$\kappa(C)=\sigma_{\max}(C)/\sigma_{\min}(C)$ the 2-norm
condition number of a matrix $C$ with $\sigma_{\max}(C)$ and $\sigma_{\min}(C)$
being the largest and smallest singular values of $C$, respectively, and $I_k$ by the $k\times k$ identity
matrix with the subscript $k$ dropped whenever it is clear.

\section{A review of the JBD method and IRJBD algorithm}\label{sec:2}
Let the thin QR factorization of $\left\{A, L\right\}$ be
\begin{equation}
\label{QR}
    \left(\begin{array}{l}
    A \\
    L
    \end{array}\right)=Q R=\left(\begin{array}{l}
    Q_A \\
    Q_L
\end{array}\right) R
\end{equation}
with $Q_A \in \mathbb{R}^{m \times n}$ and $Q_L \in \mathbb{R}^{p \times n}$. Then
GSVD (\ref{GSVD in matrix terms}) is intimately related to the CS decomposition (CSD) (cf.
\cite{GolubandVan2013matrix,book:StewartMatrixAlgorithms}) of the matrix pair $\{Q_A,\,Q_L\}$ by
\begin{equation}
\label{CS of QA and QL}
Q_A=P_AC_AW^{\mathrm T},\quad Q_L=P_LS_LW^{\mathrm T}
\end{equation}
with $W$ being orthogonal,
and
\begin{equation}
\label{X=R-1W}
X=R^{-1}W \ \mbox{or}\ W=RX.
\end{equation}

Given starting vectors $u_1$ and $\hat v_1$, the $k$-step lower and upper Lanczos bidiagonalization processes \cite{Bjorck1996LeastSquares,LSQR}
of $Q_A$ and $Q_L$ are
\begin{equation}
\label{Lanczos bidiagonalization}
    \begin{aligned}
    & Q_A V_k=U_{k+1} B_k, & & Q_A^{\mathrm T} U_{k+1}=V_k B_k^{\mathrm T}+\alpha_{k+1} v_{k+1} e_{k+1}^{\mathrm T}, \\
    & Q_L \widehat{V}_k=\widehat{U}_k \widehat{B}_k, & & Q_L^{\mathrm T} \widehat{U}_k=\widehat{V}_k \widehat{B}_k^{\mathrm T}+\hat{\beta}_k \hat{v}_{k+1} e_k^{\mathrm T},
    \end{aligned}
\end{equation}
respectively, where
$$
\resizebox{\textwidth}{!}{$
B_k=\left(\begin{array}{cccc}
\alpha_1 & & & \\
\beta_2 & \alpha_2 & & \\
& \beta_3 & \ddots & \\
& & \ddots & \alpha_k \\
& & & \beta_{k+1}
\end{array}\right) \in \mathbb{R}^{(k+1) \times k}, \quad\widehat{B}_k=\left(\begin{array}{cccc}
\hat{\alpha}_1 & \hat{\beta}_1 & & \\
& \hat{\alpha}_2 & \ddots & \\
& & \ddots & \hat{\beta}_{k-1} \\
& & & \hat{\alpha}_k
\end{array}\right) \in \mathbb{R}^{k \times k},
$}
$$
$$
U_{k+1}=\left(u_1, \ldots, u_{k+1}\right) \in \mathbb{R}^{m \times(k+1)}, \quad V_k=\left(v_1, \ldots, v_k\right) \in \mathbb{R}^{n \times k},
$$
$$
\widehat{U}_k=\left(\hat{u}_1, \ldots, \hat{u}_k\right) \in \mathbb{R}^{p \times k}, \quad \widehat{V}_k=\left(\hat{v}_1, \ldots, \hat{v}_k\right) \in \mathbb{R}^{n \times k}
$$
with $U_{k+1},\,V_k,\,\widehat{U}_k$ and $\widehat{V}_k$ being column orthonormal, and $v_1=\frac{Q_A^{\mathrm{T}}u_1}{\left\|Q_A^{\mathrm{T}}u_1\right\|}$.

One can connect the above two bidiagonalization processes
by choosing the starting vector $\hat{v}_1=v_1$ in the upper Lanczos bidiagonalization process. In this case,
it is proved in \cite{Kilmer-JBD,Zha-JBD} that
\begin{equation}
\label{relation of v and vhat}
\hat{v}_{i+1}=(-1)^i v_{i+1}, \quad \hat{\alpha}_i \hat{\beta}_i=\alpha_{i+1} \beta_{i+1}, \quad i=1,2,\dots,k-1.
\end{equation}
For $A$ and $L$ large, computing the QR factorization (\ref{QR}) is generally unaffordable. We suppose
that computing this factorization is prohibitive throughout this paper.
Without $Q_A$, $Q_L$ and $R$ at hand, the two joint Lanczos bidiagonalization processes in \eqref{Lanczos bidiagonalization} and
\eqref{relation of v and vhat} can be realized by the JBD process \cite{Jia-JBD-finite-precision,Kilmer-JBD}, as described in \Cref{alg:JBD}.

\begin{algorithm}[H]
	\renewcommand{\algorithmicrequire}{\textbf{Input:}}
	\renewcommand{\algorithmicensure}{\textbf{Output:}}
	\caption{The $k$-step JBD process of $\left\{A,L\right\}$}
	\label{alg:JBD}
	\begin{algorithmic}[1]
	    \REQUIRE $u_1 \in \mathbb{R}^m$ with $\|u_1\|=1$
        \STATE $\alpha_1 v_1^{\prime}=QQ^{\mathrm{T}}\left(\begin{array}{c}
        u_{1} \\
        0_p
        \end{array}\right)$ with the normalizing factor $\alpha_1$ such that $\|v_1^{\prime}\|=1$;
        \STATE $\hat{\alpha}_1 \hat{u}_1=v_1^{\prime}(m+1: m+p)$;
        \FOR{$i=1,2,...,k$}
		\STATE $\beta_{i+1} u_{i+1}=v_i^{\prime}(1: m)-\alpha_i u_i$;
            \label{JBD inner start}
        \STATE $\alpha_{i+1} v_{i+1}^{\prime}=QQ^{\mathrm{T}}\left(\begin{array}{c}
        u_{i+1} \\
        0_p
        \end{array}\right)-\beta_{i+1}v_i^{\prime}$ with $\|v_{i+1}^{\prime}\|=1$;
        \STATE $\hat{\beta}_i=\left(\alpha_{i+1} \beta_{i+1}\right)/\hat{\alpha}_i$;
        \STATE $\hat{\alpha}_{i+1} \hat{u}_{i+1}=(-1)^i v_{i+1}^{\prime}(m+1: m+p)-\hat{\beta}_i \hat{u}_i$.
        \label{JBD inner end}
	    \ENDFOR
	\end{algorithmic}
\end{algorithm}

Since $QQ^{\mathrm{T}}(u_i^{\mathrm{T}},\,0_p^{\mathrm{T}})^{\mathrm{T}}, i=1,2,\dots,k+1$
are the orthogonal projections of $(u_i^{\mathrm{T}},\,0_p^{\mathrm{T}})^{\mathrm{T}}$
onto the column space of $(A^{\mathrm{T}},\,L^{\mathrm{T}})^{\mathrm{T}}$, we have $QQ^{\mathrm{T}}(u_i^{\mathrm{T}},\,0_p^{\mathrm{T}})^{\mathrm{T}}=(A^{\mathrm{T}},\,L^{\mathrm{T}})^{\mathrm{T}}y_i$, where
\begin{equation}
\label{least-squares}
y_i=\arg \min \limits_{y \in \mathbb{R}^n}\left\|\left(\begin{array}{l}A \\ L\end{array}\right) y-\left(\begin{array}{l}u_i \\ 0_p\end{array}\right)\right\|,\quad i=1,\dots,k+1.
\end{equation}
Throughout the paper, these large least squares (LS) problems are supposed to
be iteratively solved by the most commonly used LSQR algorithm \cite{LSQR}, a Krylov
subspace method \cite{Bjorck1996LeastSquares}.
In such a way, we avoid the computation of the QR factorization \eqref{QR}.

Suppose that, at this moment, the LS problems in \Cref{alg:JBD} are solved accurately.
By exploiting \eqref{relation of v and vhat}, the JBD process can be written as
\begin{equation}
\label{JBD}
    \begin{aligned}
    & \left(I_m, 0_{m,p}\right) V_k^{\prime}=U_{k+1} B_k, \\
    & Q Q^{\mathrm T}\left(\begin{array}{c}
    U_{k+1} \\
    0_{p,k+1}
    \end{array}\right)=V_k^{\prime} B_k^{\mathrm T}+\alpha_{k+1} v_{k+1}^{\prime} e_{k+1}^{\mathrm T}, \\
    & \left(0_{p,m}, I_p\right) V_k^{\prime} D_k=\widehat{U}_k \widehat{B}_k,
    \end{aligned}
\end{equation}
where
\begin{equation}\label{vprime}
V_k^{\prime}=QV_k \ \ \mbox{and}\ \  D_k=\operatorname{diag}\left(1,-1, \ldots,(-1)^{k-1}\right) \in \mathbb{R}^{k \times k}.
\end{equation}
Let
\begin{equation}\label{bkbar}
H_k=R^{-1} V_k=\left(h_1, \ldots, h_k\right),\ \ \bar{B}_k=\widehat{B}_k D_k.
\end{equation}
Relations \eqref{JBD} and \eqref{bkbar} lead to
the basic relations of the JBD process:
\begin{equation}
\label{BTB+barBTbarB=I}
\begin{aligned}
&A H_k=U_{k+1} B_k,\quad LH_k=\widehat{U}_k \bar{B}_k,\\
&B_k^{\mathrm T} B_k+\bar{B}_k^{\mathrm T} \bar{B}_k=I_k,
\end{aligned}
\end{equation}
which generates the explicit
left Krylov subspaces $\operatorname{span}(U_{k+1})$ and $\operatorname{span}(\widehat U_k)$ and the
implicit right subspace $\operatorname{span}(H_k)=R^{-1}\operatorname{span}(V_k)$ with $R$ being
unavailable and  $\operatorname{span}(V_k)$ being a Krylov subspace, respectively: it is known from \eqref{Lanczos bidiagonalization} that
\begin{eqnarray*} &\operatorname{span}(V_k)=\operatorname{span}(\widehat{V}_k)=\mathcal{K}_k(Q_A^{\rm T}Q_A,v_1)=\mathcal{K}_k(Q_L^{\rm T}Q_L,v_1),  \\
&{\rm span}(U_{k+1})=\mathcal{K}_{k+1}(Q_AQ_A^{\rm T},u_1),\
{\rm span}(\widehat{U}_k)=\mathcal{K}_k(Q_LQ_L^{\rm T},\hat{u}_1),
\end{eqnarray*}
where $\mathcal{K}_k(C,w)$ is the $k$-dimensional Krylov subspace generated by the matrix $C$
and the starting vector $w$.

With the above left and right subspaces, the JBD method falls into the category of
the framework of general Rayleigh--Ritz projection methods \cite{JiaandZhang2023FEASTGSVD}
for the GSVD problem; it obtains the projection matrix pair $\{B_k,\,\bar B_k\}$,
and computes the approximate GSVD components, i.e., Ritz approximations,
as follows: Let the SVDs of $B_k$ and $\bar B_k$, or equivalently, the CSD
of $\{B_k,\,\bar B_k\}$ be
\begin{equation}
\label{GSVD of Bk and Bkbar}
\begin{aligned}
    B_k=\widetilde Q_k^A \widetilde C_k \widetilde W_k^{\mathrm T}, \quad \widetilde C_k=\text{diag}\left(\tilde c_1, \dots, \tilde c_k\right), \quad 1> \tilde c_1>\cdots> \tilde c_k> 0,\\
    \bar{B}_k=\widetilde Q_k^L \widetilde S_k \widetilde W_k^{\mathrm T}, \quad \widetilde S_k=\text{diag}\left(\tilde s_1, \dots, \tilde s_k\right), \quad 0< \tilde s_1<\cdots<\tilde s_k< 1,
\end{aligned}
\end{equation}
where $\widetilde Q_k^A$ are orthonormal, $\widetilde Q_k^L$ and $\widetilde W_k$ are orthogonal, and $\tilde
c_i^2+\tilde s_i^2=1$. Then the Ritz approximations are
$\{\tilde c_i,\tilde s_i,\tilde x_i,\tilde p_i^A,\tilde p_i^L\}$, where
\begin{equation}
\label{approximate generalized singular vectors from JBD}
\tilde x_i=R^{-1}V_k\tilde w_i,\quad\tilde p_i^A=U_{k+1}\tilde q_i^A, \text{and}\quad \tilde p_i^L=\widehat U_k\tilde q_i^L,
\end{equation}
with $\tilde w_i$, $\tilde q_i^A$ and $\tilde q_i^L$ being the $i$-th columns of $\widetilde W_l$, $\widetilde Q_l^A$ and $\widetilde Q_l^L$, respectively.
We call the $\tilde c_i/\tilde s_i$ Ritz values, $\tilde
x_i$ the right Ritz vectors, and $\tilde p_i^A$ and $\tilde p_i^L$ the left Ritz vectors
for $A$ and $L$, respectively, where, by exploiting $V^\prime=QV_k$ in \eqref{vprime},
$\tilde x_i$ is computed by solving the consistent linear system
\begin{equation}
\label{calculate x_i by lsqr}
\left(\begin{array}{c}
A \\
L
\end{array}\right) \tilde x_i=Q R R^{-1} V_k \tilde w_i=V^{\prime}_k \tilde w_i
\end{equation}
using the LSQR algorithm.  From the first two relations in
\eqref{BTB+barBTbarB=I}, we have $A \tilde x_i=\tilde c_i \tilde p_i^A$ and
$L \tilde x_i=\tilde s_i \tilde p_i^L$. The JBD
method then uses the $l$ largest or smallest Ritz approximations to approximate the $l$ largest
or smallest GSVD components of $\{A,L\}$.

It has been proved in \cite{Jia2020RegularizationofLSQR} that, provided that the Lanczos
bidiagonalization process does not break down, the Ritz values obtained by the Lanczos
bidiagonalization method are always simple,
and the method works as if the underlying matrix had only {\em simple nonzero}
singular values and {\em multiple} zero singular values;
the Ritz values approach {\em distinct} nonzero singular values of the original
matrix and the Ritz vectors approach the
{\em specific} singular vectors in the corresponding singular subspaces of the
original matrix. Suppose that the JBD process does not
break down before step $k$. According to the equivalence of the JBD process of $\{A,L\}$ and
the Lanczos bidiagonalization processes of $\{Q_A,Q_L\}$, the JBD method works as if the
finite {\em nonzero} generalized singular
values of $\{Q_A,Q_L\}$ were simple; if the matrix pair has zero and infinite generalized
singular values, then the JBD method cannot compute them. As a result, the Ritz values $\{\tilde
c_i,\tilde s_i\}$ are always simple and approach {\em distinct} nontrivial
generalized singular values of $\{A,L\}$, and
the right and left Ritz vectors $\tilde x_i$, $\tilde p^A_i$ and $\tilde p^L_i$ approach the
corresponding specific right and left generalized
singular vectors in the corresponding right and left generalized singular subspaces. 
Therefore, in the sequel, for the JBD method,
we suppose that the generalized singular values of $\{A,L\}$ are
simple. From (\ref{GSVD in matrix terms}) and (\ref{GSVD in vector terms})
we label them as
\begin{equation}  \label{labelrule}
1=c_0>c_{1}>c_1>\cdots>c_q>c_{q+1}=0,
\end{equation}
where $c_0$ and $c_{q+1}$ may be nonexistent.

In terms of (\ref{GSVD in vector terms}), the GSVD residual of a general approximate GSVD component,
e.g., the Ritz approximation $\{\tilde c_i,\tilde s_i,\tilde x_i,\tilde p_i^A,\tilde p_i^L\}$, is defined by
\begin{equation}
\label{residual}
r\left(\tilde c, \tilde s, \tilde x, \tilde p^A, \tilde p^L\right)=\left(\begin{array}{c}
A \tilde x_i-\tilde c_i \tilde p_i^A \\
L \tilde x_i-\tilde s_i \tilde p_i^L \\
\tilde s_i A^{\mathrm{T}} \tilde p_i^A-\tilde c_i L^{\mathrm{T}} \tilde p_i^L
\end{array}\right).
\end{equation}

\section{Convergence analysis on the JBD method and the Rayleigh--Ritz projection method}\label{sec:3}

A subclass of general Rayleigh--Ritz projection methods proposed
\cite{JiaandZhang2023FEASTGSVD} for the GSVD problem
is as follows: Write a {\em general} right subspace ${\rm span}(H_k)=R^{-1}{\rm span}(V_k)$
and two {\em general} left subspaces ${\rm span}(U_{k+1})$ and ${\rm span}(\widehat{U}_k)$
for $A$ and $L$, and suppose that the approximate
GSVD components satisfy $A \tilde x_i=\tilde c_i \tilde p_i^A$ and
$L \tilde x_i=\tilde s_i \tilde p_i^L$, which the JBD method, the JD GSVD
method \cite{Huang-Numerical-experiments} and the CJ-FEAST GSVDsolver \cite{JiaandZhang2023FEASTGSVD}
satisfy. Then
the projection matrix pair is $\{U_{k+1}^{\rm T}A H_k,\widehat{U}_k^{\rm T} LH_k\}$, and it is equal to
$\{B_k,\bar{B}_k\}$ in the JBD method.
We need to mention that a general
Rayleigh--Ritz projection method in \cite{JiaandZhang2023FEASTGSVD}
requires that the dimensions of two left subspaces are no smaller
than that of the right subspace \cite{JiaandZhang2023FEASTGSVD} but does not require
that $A \tilde x_i=\tilde c_i \tilde p_i^A$ and
$L \tilde x_i=\tilde s_i \tilde p_i^L$.

In this section, we analyze the convergence of a general Rayleigh--Ritz projection method
\cite{JiaandZhang2023FEASTGSVD} that satisfies
$A \tilde x_i=\tilde c_i \tilde p_i^A$ and
$L \tilde x_i=\tilde s_i \tilde p_i^L$, and particularly focus on
the JBD method. We drop the
subscripts of GSVD components and approximate GSVD components for brevity, and
let $\{c,s,x,p^A,p^L\}$ be a
simple GSVD component of $\{A,L\}$ and $\{\tilde c,\tilde s,\tilde x,\tilde p^A,\tilde p^L\}$ be its
Ritz approximation. From \eqref{CS of QA and QL} and \eqref{X=R-1W}
we know that $Rx$ is
the corresponding right singular vector of $Q_A$ and $Q_L$
associated with the singular values $c$ and $s$, respectively.
Relation \eqref{approximate generalized singular vectors from JBD}
shows that $R\tilde x$ is a unit-length approximation of $Rx$ extracted from a general right subspace
$\operatorname{span}(V_k)$ for the SVDs of $Q_A$ and $Q_L$.

We measure the deviation of $Rx$ from $\operatorname{span}(V_k)$ by
\begin{equation}\label{epsilonk}
\epsilon=\sin\angle\left(Rx,V_k\right),
\end{equation}
the sine of the acute angle between $Rx$ and $\operatorname{span}(V_k)$.
Notice from \eqref{QR} that $R^{\rm T}R=A^{\rm T}A+L^{\rm T}L=: M$ is symmetric
positive definite.
It is seen from \eqref{bkbar} that
$$
\epsilon=\sin\angle (x,H_k)_M,
$$
which is the deviation of $x$ from the right subspace ${\rm span}(H_k)$ in the $M$-inner product.
Particularly, $\epsilon=0$
means that $Rx\in\operatorname{span}(V_k)$, i.e., a perfect right subspace. The hypothesis $\epsilon\rightarrow0$ is necessary for the convergence of any projection method because, for
 {\em any} approximation $\tilde x \in R^{-1}\operatorname{span}(V_k)$, we have
$$
\sin\angle(\tilde x,x)_M=\sin\angle(R\tilde x,Rx)\geqslant\epsilon,
$$
which means that if $\epsilon\not\rightarrow0$ then
$\tilde x$ definitely does not converge to $x$.

How to construct a sequence of $\operatorname{span}(V_\epsilon):=\operatorname{span}(V_k)$ with fixed or
varying dimension such that
$\epsilon\rightarrow 0$ is another important issue that is not concerned in this paper.
For instance, in the CJ-FEAST GSVDsolver \cite{JiaandZhang2023FEASTGSVD},
a sequence of right subspaces ${\rm span}(H_\epsilon)=R^{-1}{\rm span}(V_\epsilon)$
with fixed dimension is iteratively constructed, and those $\epsilon$ for the
concerning right generalized singular vectors $x$ have been proved to tend to
zero under certain mild conditions as the iterations proceed. In what follows,
whenever necessary, one can replace
the subscripts $k$ and $k+1$ by $\epsilon$ so as to indicate the dependence of subspaces
and projection matrices on $\epsilon$.

Since the proofs and analysis of the following convergence results only require that $\epsilon\rightarrow 0$, they hold for the subclass of general Rayleigh--Ritz projection methods
\cite{JiaandZhang2023FEASTGSVD} satisfying $A \tilde x=\tilde c \tilde p^A$ and
$L \tilde x=\tilde s \tilde p^L$
and are thus general. The
changes are that (i) the projection matrices $B_k=U_{k+1}^TA H_k$ and $\bar{B}_k=\widehat{U}_k LH_k$ in (\ref{BTB+barBTbarB=I}) are no longer lower and upper bidiagonal and (ii) only the concerning
generalized singular value $c/s$ is simple and there is no requirement on the multiplicities of other generalized singular values.

The first key step toward a convergence analysis is to transform the GSVD problem and the Rayleigh--Ritz projection method for it into a standard eigenvalue problem and the corresponding Rayleigh--Ritz projection method for it, respectively. Define
\begin{equation}\label{matrixH}
H=Q_A^{\mathrm T}Q_A-Q_L^{\mathrm T}Q_L.
\end{equation}
It is known from \eqref{CS of QA and QL} and \eqref{X=R-1W} that $RX=W$ is orthogonal
and the eigendecomposition of $H$ is
\begin{equation}
\label{eigendecomposition of H}
H=W\left(C_A^{\mathrm T}C_A-S_L^{\mathrm T}S_L\right)W^{\mathrm T},
\end{equation}
and from \eqref{GSVD of Bk and Bkbar} that the eigenvalues and eigenvector matrix of
$B_k^{\mathrm T}B_k-\bar B_k^{\mathrm T}\bar B_k$ are the $\tilde c^2-\tilde s^2$ and $\widetilde W_k$.
Then the Rayleigh--Ritz projection method for the GSVD problem is
equivalent to the Rayleigh--Ritz projection method for the eigenvalue problem of $H$:
the projection
subspace $\operatorname{span}(V_k)$  with
$V_k$ column orthonormal, the projection matrix
$$
V_k^{\mathrm T}HV_k=
V_k^{\rm T}Q_A^{\rm T}Q_AV_k-V_k^{\rm T}Q_L^{\rm T}Q_LV_k=B_k^{\mathrm T}B_k-\bar B_k^{\mathrm T}\bar B_k,
$$
the Ritz values $\tilde c^2-\tilde s^2$, and
the Ritz vectors $R\tilde x=V_k\tilde w$
with $\tilde{w}$ being the columns of $\widetilde W_k$.  For the JBD method,
the projection matrix $B_k^{\mathrm T}B_k-\bar B_k^{\mathrm T}\bar B_k$ is symmetric tridiagonal 

\subsection{Convergence of the Ritz values}\label{subsec:3.1}

With the preceding equivalence transformations of the problems and the Rayleigh--Ritz methods,
Theorem 4.1 of \cite{JiaandStewart2001} leads to the following theorem directly.

\begin{theorem}\label{thm:ritzvalue}
Let $\delta=\frac{\epsilon}{\sqrt{1-\epsilon^2}}\|H\|$. There is a matrix $E_k$ satisfying $\left\|E_k\right\|\leqslant\delta$ such that $c^2-s^2$ is an eigenvalue of the perturbed matrix $B_k^{\mathrm T}B_k-\bar B_k^{\mathrm T}\bar B_k+E_k$.
\end{theorem}

This theorem indicates that, as $\epsilon\rightarrow 0$, there is an eigenvalue of $B_k^{\mathrm T}B_k-\bar B_k^{\mathrm T}\bar B_k$ that converges to $c^2-s^2$.
Moreover, from (4.3) and (4.4) in \cite{JiaandStewart2001}, we obtain
\begin{equation}
\label{tmp202502031044}
\left\|\left(B_k^{\mathrm T} B_k-\bar{B}_k^{\mathrm T} \bar{B}_k\right) \frac{V_k^{\mathrm T} Rx}{\left\|V_k^{\mathrm T} Rx\right\|}-\left(c^2-s^2\right) \frac{V_k^{\mathrm T} Rx}{\left\|V_k^{\mathrm T} Rx\right\|}\right\| \leqslant\delta.
\end{equation}
Regard $(c^2-s^2,\frac{V_k^{\mathrm T} Rx}{\|V_k^{\mathrm T} Rx\|})$ as an approximate eigenpair of $B_k^{\mathrm T}B_k-\bar B_k^{\mathrm T}\bar B_k$. Then \eqref{tmp202502031044} establishes a compact upper bound for its residual norm. Thus, by the standard perturbation theory (cf. Corollary 3.3 in \cite[p.61]{book:SaadEigenvalue}), we have the following error bound.
\begin{theorem}
\label{th: eigenvalue of BkTBk-barBkTbarBk}
There is an eigenvalue $\tilde c^2-\tilde s^2$ of $B_k^{\mathrm T}B_k-\bar B_k^{\mathrm T}\bar B_k$ such that
\begin{equation}
\label{tmp202502031059}
\left|\tilde c^2-\tilde s^2-(c^2-s^2)\right|\leqslant\delta,
\end{equation}
meaning that $\tilde c^2-\tilde s^2\rightarrow c^2-s^2$ as $\epsilon\rightarrow 0$.
\end{theorem}

We present the following interlacing
property of generalized singular values and Ritz values obtained by the JBD method.

\begin{theorem}\label{interlacetheorem}
Suppose \cref{alg:JBD} does not break down before step $k$ and
the nontrivial generalized singular values $\{\frac{c_i}{s_i}\}_{i=1}^n$ of $\{A,L\}$ are simple and labeled as \eqref{labelrule}, and
label the $k$ Ritz values $\{\frac{\tilde c_i}{\tilde s_i}\}_{i=1}^k$ in the same manner. Then
\begin{equation}
\label{interlacing theorem eq1}
\frac{c_{i+\left(n-k\right)}}{s_{i+\left(n-k\right)}}<\frac{\tilde c_i}{\tilde s_i}<\frac{c_i}{s_i},\quad i=1,\dots,k.
\end{equation}
If the matrix pair $\{A,L\}$ has $q_1$ zero and $q_2$
infinite generalized singular values, written as
$\{c_0,s_0\}=\{1,0\}$ and $\{c_{q+1},s_{q+1}\}=\{0,1\}$, respectively, then
\begin{equation}
\label{interlacing theorem eq2}
\frac{c_{i+\left(q+1-k\right)}}{s_{i+\left(q+1-k\right)}}<\frac{\tilde c_i}{\tilde s_i}<\frac{c_{i-1}}{s_{i-1}},\quad i=1,\dots,k,
\end{equation}
where $q=n-q_1-q_2$ {\rm (cf. \eqref{GSVD in matrix terms} and \eqref{GSVD in vector terms})}.
\end{theorem}

\begin{proof}
Notice that $V_k^{\mathrm T}HV_k=B_k^{\mathrm T}B_k-\bar B_k^{\mathrm T}\bar B_k$ is an unreduced symmetric tridiagonal matrix. By the eigenvalue strict interlacing theorem \cite[p.203]{Parlett1980symmetric} of symmetric matrices, for the first case, the eigenvalues $\{c_i^2-s_i^2\}_{i=1}^n$ and $\{\tilde c_i^2-\tilde s_i^2\}_{i=1}^k$ of $H$ and $B_k^{\mathrm T}B_k-\bar B_k^{\mathrm T}\bar B_k$ satisfy
$$
c_{i+\left(n-k\right)}^2-s_{i+\left(n-k\right)}^2<\tilde c_i^2-\tilde s_i^2< c_i^2- s_i^2,\quad i=1,\dots,k,
$$
which establishes \eqref{interlacing theorem eq1} by a simple manipulation;
for the second case, changing
the subscripts $i$ and $n$ in the lower and upper bounds of (\ref{interlacing theorem eq1}) to $i-1$
and $q+1$, respectively, we obtain
$$
c_{i+\left(q+1-k\right)}^2-s_{i+\left(q+1-k\right)}^2<\tilde c_i^2-\tilde s_i^2< c_{i-1}^2- s_{i-1}^2,\quad i=1,\dots,k,
$$
which leads to (\ref{interlacing theorem eq2}).
\end{proof}

\begin{remark}
This theorem indicates that (i)
the largest Ritz values $\frac{\tilde c}{\tilde s}$ approach the corresponding largest distinct
generalized singular values $\frac cs$ from below, and the signs in the absolute values in
the left-hand side of (\ref{tmp202502031059}) are always negative, and (ii)
the smallest Ritz values approach the corresponding smallest distinct
generalized singular values from above, and the signs in the absolute values of the left-hand side of (\ref{tmp202502031059}) are always positive.
\end{remark}

\begin{remark}
For the general Rayleigh--Ritz projection method, the projection matrix $V_k^THV_k$ is symmetric but no longer tridiagonal. The interlacing properties \eqref{interlacing theorem eq1} and \eqref{interlacing theorem eq2} still hold with the strict inequality signs $<$ in them replaced by  inequality signs $\leq$ when assuming that the method can compute generalized
singular values counting multiplicities.
Particularly, this leads to a general interlacing property of the generalized
singular values of $\{A,L\}$ and those of the submatrix pair consisting of arbitrary $k$
columns of $\{A,L\}$ when taking $V_k$ to be a matrix consisting of arbitrary columns of
the identity matrix $I_n$. As a matter of fact, it is readily justified that
such general interlacing property
does not require to assume that the generalized singular values are simple.
\end{remark}

Next we take a closer look at (\ref{tmp202502031059}) and get insight into the convergence of the Ritz
values $\frac{\tilde{c}}{\tilde{s}}$ themselves rather than the differences
$\tilde{c}^2-\tilde{s}^2$. The chordal metric
$$\chi\left(\lambda,\mu\right)=\frac{\left|\lambda-\mu\right|}{\sqrt{1+|\lambda|^2}\sqrt{1+|\mu|^2}}$$
for two nonzero scalars $\lambda$ and $\mu$ is a standard gauge
to measure errors of approximate generalized
eigenvalues of a regular matrix pair \cite[Ch.15]{Parlett1980symmetric} and
\cite[Ch.2]{book:StewartMatrixAlgorithms}. In what follows we study
the convergence of the Ritz values
using the chordal metric.

\begin{theorem}
\label{Th:error of generalized singular value}
For the general Rayleigh--Ritz projection method, there is a Ritz value, i.e., generalized singular value $\frac{\tilde c}{\tilde s}$ of the matrix pair
$\{V_k^{\rm T}Q_A^{\rm T}Q_AV_k,V_k^{\rm T}Q_L^{\rm T}Q_LV_k \}$, such that
\begin{equation}
\label{error of square of generalized singular value}
\chi\left(\frac{\tilde c^2}{\tilde s^2},\frac{c^2}{s^2}\right)\leqslant\delta,
\end{equation}
where $\delta$ is defined in \cref{thm:ritzvalue}.
For a nonzero finite generalized singular value $\frac{c}{s}$, if $\epsilon$
defined by \eqref{epsilonk} is small enough such that $\delta\leqslant \min\{2c^2,2s^2\}$, then
\begin{equation}
\label{error of nontrival generalized singular value}
\chi\left(\frac{\tilde c}{\tilde s}, \frac{c}{s}\right)=\frac{\delta}{4 c s}+O\left(\delta^2\right).
\end{equation}
\end{theorem}

\begin{proof}
By definition, we have
\begin{equation}\label{cs2}
\chi\left(\frac{\tilde c^2}{\tilde s^2},\frac{c^2}{s^2}\right) =\frac{\left|\tilde{c}^2 s^2-c^2 \tilde{s}^2\right|}{\sqrt{\tilde{s}^4+\tilde{c}^4} \sqrt{s^4+c^4}}.
\end{equation}
By $\tilde s^2=1-\tilde c^2$ and $s^2=1-c^2$, relation
(\ref{tmp202502031059}) is equivalent to $\left|\tilde c^2-c^2\right|\leqslant \frac12\delta$. Thus,
$$
\left|\tilde{c}^2 s^2-c^2 \tilde{s}^2\right| =\left|\tilde{c}^2 \left(1-c^2\right)-c^2 \left(1-\tilde c^2\right)\right|=\left|\tilde{c}^2-c^2\right|=\left|\tilde{s}^2-s^2\right|\leqslant\frac12\delta.
$$
On the other hand,
$$
\begin{aligned}
\sqrt{\tilde s^4+\tilde c^4}\geqslant \frac{\sqrt2}{2}\left(\tilde s^2+\tilde c^2\right)=\frac{\sqrt2}{2},\\
\sqrt{ s^4+ c^4}\geqslant \frac{\sqrt2}{2}\left( s^2+ c^2\right)=\frac{\sqrt2}{2}.
\end{aligned}
$$
Substituting the above three relations into \eqref{cs2} establishes \eqref{error of square of generalized singular value}.

To prove \eqref{error of nontrival generalized singular value}, notice that $\tilde c$ and $\tilde s$ are
one-to-one correspondence. By definition, we obtain
$$
\chi(\frac{\tilde{c}}{\tilde{s}},
\frac{c}{s})=|\tilde{c} s-\tilde{s} c|.
$$
Regard it as a function of $\tilde c$. Then $\tilde{c} s-\tilde{s} c$ is
increasing with respect to $\tilde c$. From $|\tilde c^2-c^2|\leqslant \frac12\delta$,
it is known that
$\tilde c$ lies in the intervals
$\left[\sqrt{c^2-\frac12\delta},\sqrt{c^2+\frac12\delta}\right]$.
The assumption $\delta\leqslant \min\{2c^2,2s^2\}$ ensures that the square roots are between $0$ and $1$. Therefore, $\chi(\frac{\tilde{c}}{\tilde{s}}, \frac{c}{s})$ attains
the maximum at $\tilde c=\sqrt{c^2-\frac12\delta}$ or $\tilde c=\sqrt{c^2+\frac12\delta}$, and
$$
\begin{aligned}
\chi\left(\frac{\tilde{c}}{\tilde{s}}, \frac{c}{s}\right) & =\left|\tilde{c} s-\tilde{s} c\right| \\
& \leqslant \max \left\{\left|s \sqrt{c^2-\frac{1}{2} \delta}-c \sqrt{s^2+\frac{1}{2} \delta}\right|,\,\,\left|s \sqrt{c^2+\frac{1}{2} \delta}-c \sqrt{s^2-\frac{1}{2} \delta}\right|\right\} \\
& =\max \left\{c \sqrt{s^2+\frac{1}{2} \delta}-s \sqrt{c^2-\frac{1}{2} \delta},\,\, s \sqrt{c^2+\frac{1}{2} \delta}-c \sqrt{s^2-\frac{1}{2} \delta}\right\}.
\end{aligned}
$$
Using the Taylor expansions
$$
\sqrt{s^2+\frac{1}{2} \delta}=s+\frac{\delta}{4 s}+O\left(\delta^2\right)\quad \text{and}\quad\sqrt{c^2-\frac{1}{2} \delta}=c-\frac{\delta}{4 c}+O\left(\delta^2\right),
$$
we obtain
$$
c \sqrt{s^2+\frac{1}{2} \delta}-s \sqrt{c^2-\frac{1}{2} \delta}=\frac{\delta}{4 c s}+O\left(\delta^2\right).
$$
Similarly,
$$
s \sqrt{c^2+\frac{1}{2} \delta}-c \sqrt{s^2-\frac{1}{2} \delta}=\frac{\delta}{4 c s}+O\left(\delta^2\right).
$$
Thus (\ref{error of nontrival generalized singular value}) holds.
\end{proof}

\begin{remark}
If $\{A,L\}$ has infinite or zero generalized singular values, \eqref{error of square of generalized singular value} indicates that the largest or smallest Ritz value converges to infinite or zero
generalized singular value when the corresponding $\epsilon$ tends to zero, but \eqref{error of nontrival
generalized singular value} holds only for a nonzero finite generalized singular value.
\end{remark}

\begin{remark}
For the JBD method, it is known from \cite{Bjorck1996LeastSquares,Jia-IRRBL,Parlett1980symmetric} that
the $\epsilon$ defined by \eqref{epsilonk}
generally tend to zero first as $k$ increases
for the the right singular vectors corresponding to the extreme singular values
of $Q_A$ and $Q_L$. This means that the $\epsilon\rightarrow 0$ generally first for the right
generalized singular vectors corresponding to the extreme generalized singular values
of $\{A,L\}$. Therefore, the JBD method generally favor the extreme generalized
singular values as $k$ increases.
\end{remark}


\subsection{Convergence of the right Ritz vectors}\label{subsec:3.2}

Recall definition \eqref{matrixH} of $H$ and its eigendecomposition
\eqref{eigendecomposition of H}. Let $(Rx,RX_{\perp})$ be orthogonal, and define the matrix
\begin{equation}
\label{definition of Y_perp}
H_\perp=X_{\perp}^{\mathrm T} R^{\mathrm T}H R X_{\perp}.
\end{equation}
Then the eigenvalues of $H_\perp$ are the other eigenvalues of $H$ than $c^2-s^2$. Define
the separation of a desired eigenvalue $c^2-s^2$ and the spectrum of $H_\perp$ as
$$
\operatorname{sep}\left(c^2-s^2, H_\perp\right)=\sigma_{\min} \left(H_\perp-\left(c^2-s^2\right)I\right).
$$
Then, when
$\frac cs$ is a simple generalized singular value of $\{A,L\}$, $c^2-s^2$ is not an eigenvalue of $H_\perp$,
meaning $\operatorname{sep}(c^2-s^2, H_\perp)>0$. Theorem 3.1 in \cite{JiaandStewart2001} leads to the
following error bound for the Ritz vector $\tilde x$.

\begin{theorem}
\label{Th:error of Rx bounded by residual}
Suppose that  $(\tilde c^2-\tilde s^2,R\tilde x)$ is an approximation to
the eigenpair $(c^2-s^2,Rx)$ of $H$, and denote by
$\rho=\left\|H R \tilde{x}-\left(\tilde{c}^2-\tilde{s}^2\right) R \tilde{x}\right\|$
its residual norm. If
\begin{equation}
\label{assumption of sep(c^2-s^2,Yperp)}
\operatorname{sep}\left(c^2-s^2, H_\perp\right)-\left|\tilde c^2-\tilde{s}^2-\left(c^2-s^2\right)\right|>0,
\end{equation}
then
$$
\begin{aligned}
\sin\angle(\tilde x,x)_M &\leqslant\frac{\rho}{\operatorname{sep}\left(\tilde{c}^2-\tilde s^2,H_\perp\right)} \\
&\leqslant \frac{\rho}{\operatorname{sep}\left(c^2-s^2, H_\perp\right)-\left|\tilde c^2-\tilde{s}^2-\left(c^2-s^2\right)\right|}.
\end{aligned}
$$
\end{theorem}

\Cref{th: eigenvalue of BkTBk-barBkTbarBk} has shown that $|\tilde c^2-\tilde s^2-(c^2-s^2)|\rightarrow0$ as
$\epsilon\rightarrow0$. Since $\frac cs$ is supposed to be
a simple generalized singular value of $\{A,L\}$, \eqref{assumption of sep(c^2-s^2,Yperp)} must hold for $|\tilde c^2-\tilde s^2-(c^2-s^2)|$ sufficiently small.
\Cref{Th:error of Rx bounded by residual} indicates that if the residual norm $\rho$ tends to $0$ then
$\sin\angle(\tilde x,x)_M\rightarrow0$; that is, the right Ritz vector $\tilde x\rightarrow x$,
and its error in the $M$-inner product
can be measured in terms of the computable residual norm $\rho$ and the separation of $c^2-s^2$ from
the other eigenvalues of $H$.

It is instructive to see how the right Ritz vector $\tilde x$
can fail to converge even if $\epsilon\rightarrow 0$.
Notice that $R\tilde{x}=V_k\tilde{w}$
with $\tilde{w}$ being the
eigenvector of $V_k^{\rm T}HV_k=V_k^{\rm T}Q_A^{\rm T}Q_AV_k-V_k^{\rm T}Q_L^{\rm T}Q_LV_k$
corresponding to the eigenvalue $\tilde c^2-\tilde
s^2$. Let $(\tilde w,\widetilde W_{\perp})$ be orthogonal, and define the matrix
$$
C_k=\widetilde{W}_{\perp}^{\mathrm T}(V_k^{\rm T}Q_A^{\rm T}Q_AV_k-V_k^{\rm T}Q_L^{\rm T}Q_LV_k) \widetilde{W}_\perp.
$$
Then we have the following result, which is adapted from Theorem 5.1 of \cite{JiaandStewart2001}.

\begin{theorem}
\label{Th:error of right generalized singular vector}
If
$\epsilon$ is small enough such that
$$
\operatorname{sep}\left(\tilde c^2-\tilde s^2, C_k\right)-\left|\tilde c^2-\tilde s^2-\left(c^2-s^2\right)\right|>0,
$$
then
$$
\sin \angle\left(\tilde x, x\right)_M
\leqslant\left(1+\frac{\|H\|}{\sqrt{1-\epsilon^2} \left(\operatorname{sep}\left(\tilde c^2-\tilde s^2, C_k\right)-\left|\tilde c^2-\tilde s^2-\left(c^2-s^2\right)\right|\right)}\right) \epsilon.
$$
\end{theorem}
Thus, a sufficient condition for $\tilde x$ to converge requires that $\operatorname{sep}(\tilde c^2-\tilde s^2,Z_\perp)$ be uniformly bounded from below by a positive constant.

However, the assumption that $\frac cs$ is simple does {\em not}
ensure that $\tilde c^2-\tilde s^2$ is
well separated from the spectrum of $C_k$, i.e., the other Ritz values than. According to the analysis in
\cite{Jia2004,JiaandStewart2001},
for a general projection subspace ${\rm span}(V_k)$, it is possible that
$\operatorname{sep}\left(\tilde{c}^2-\tilde{s}^2, C_k\right)=0$ or arbitrarily small. Some examples
are constructed to illustrate such phenomena in \cite{Jia2004}.
In the first case, $\tilde x$ either may be not unique and there are {\em more than one} Ritz vectors
that approximate the same right generalized singular vector $x$, leading to
the failure of the method; in the second case, $\tilde x$,
though unique, may have very poor accuracy or even has no accuracy.

\subsection{Convergence of the left Ritz vectors}\label{subsec:3.3}

Regarding the convergence of the left Ritz vectors, the authors in \cite{IRJBD} have
proved that if $c$ or $s$ is $0$,
which corresponds to the trivial zero or infinite
generalized singular value, the left Ritz vector
$\tilde p^A$ or $\tilde p^L$ does {\em not} converge. For a nontrivial finite generalized singular value $\frac c s$, Theorem 4.2 in \cite{JiaandZhang2023FEASTGSVD} establishes the error bounds
\begin{equation}
\label{error of left generalized singular vectors}
\begin{aligned}
& \sin \angle\left(\tilde p^A, p^A\right) \leqslant \frac{\left\|Q_A\right\|}{\tilde c} \sin \angle\left(\tilde x, x\right)_M, \\
& \sin \angle\left(\tilde p^L, p^L\right) \leqslant \frac{\left\|Q_L\right\|}{\tilde s} \sin \angle\left(\tilde x, x\right)_M.
\end{aligned}
\end{equation}
Since $\|Q_A\|\leqslant 1$ and $\|Q_L\|\leqslant 1$, the above two bounds indicate
that the left Ritz vectors
$\tilde p^A$ and $\tilde p^L$ converge if $\tilde x$ does
since $\tilde c\rightarrow c$ and $\tilde s\rightarrow s$,
though the speed of convergence could be slow if $c$ or $s$ is small.

\section{The RJBD method}\label{sec:4}

We have shown that the right and left Ritz vectors obtained by the JBD method
may converge erratically and even may fail to converge, implying that the residual norms
may converge irregularly or may not converge to zero.
Such possible irregular convergence or non-convergence
affects the convergence of a restarted JBD
algorithm whose restarted subspaces are commonly constructed by the currently available Ritz vectors that
are used to approximate desired generalized singular vectors. Given the unconditional convergence of
Ritz values,
it is thus important and attractive to retain the Ritz values but seek new better
approximate generalized singular vectors to guarantee
their convergence as $\epsilon\rightarrow 0$, which can be used to construct better restarted
subspaces so that more accurate approximate GSVD components can be extracted from them.

It is known that the refined Rayleigh--Ritz projection method
\cite{book:BaiEigenvalue,book:StewartMatrixAlgorithms,book:VanDerVorst2002} for the eigenvalue
problem can overcome the possible irregular convergence and even non-convergence of
Ritz vectors. In this section, we extend this class of projection methods
to the GSVD problem.  Particularly,
based on the JBD process, we will propose the RJBD method that
computes mathematically different approximate
right and left generalized singular vectors, called the right and left refined Ritz vectors, and prove the
unconditional convergence of these new approximations as $\epsilon\rightarrow 0$.

Exploiting \eqref{BTB+barBTbarB=I} and \eqref{approximate generalized singular vectors from JBD},
we introduce the right and left refined Ritz vectors as follows.
\begin{definition}
\label{Def: refined vectors}
For a given Ritz value $\frac{\tilde c}{\tilde s}$, i.e., the generalized singular value of $\{B_k,\bar B_k\}$, the right refined Ritz vector is defined as $\hat x=H_k\hat w\in
{\rm span}(H_k)$ with $\hat{w}$ satisfying the
residual optimality
\begin{equation}
\label{definition of refined right singular vector}
\left\|\tilde s^2 Q_A^{\mathrm{T}} Q_A V_k \hat{w}-\tilde c^2 Q_L^{\mathrm{T}} Q_L V_k \hat{w}\right\|=\min _{\|w\|=1}\left\|\tilde s^2 Q_A^{\mathrm{T}} Q_A V_k w-\tilde c^2 Q_L^{\mathrm{T}} Q_L V_k w\right\|,
\end{equation}
and the left refined Ritz vectors associated with $A$ and $L$ are defined as
\begin{equation}
\label{definition of refined left singular vectors}
\hat p^A=\frac{A\hat x}{\|A\hat x\|}=\frac{U_{k+1}B_k\hat w}{\|B_k\hat w\|}=:U_{k+1}\hat q^A,\quad \hat p^L=\frac{L\hat x}{\|L\hat x\|}=\frac{\widehat{U}_k\bar B_k\hat w}{\|\bar B_k\hat w\|}=:\widehat U_k\hat q^L.
\end{equation}
\end{definition}

The following theorem shows what $\hat{w}$ is and how to  compute it efficiently and accurately.
\begin{theorem}
\label{Th: computation of refined right singular vector}
The vector $\hat{w}$ in \eqref{definition of refined right singular vector} is the right singular vector of $\left(\begin{array}{c}
\tilde s^2 B_k^{\mathrm{T}} B_k-\tilde c^2 \bar{B}_k^{\mathrm{T}} \bar{B}_k \\
\alpha_{k+1} \beta_{k+1} e_k^{\mathrm{T}}
\end{array}\right)\in \mathbb{R}^{(k+1)\times k}$ corresponding to its smallest singular value
$\sigma_{\min}(\cdot)$.
\end{theorem}

\begin{proof}
Exploiting \eqref{Lanczos bidiagonalization} and the third relation in \eqref{JBD},
we have
$$
\begin{aligned}
&\left\|\tilde s^2 Q_A^{\mathrm{T}} Q_A V_k \hat{w}-\tilde c^2 Q_L^{\mathrm{T}} Q_L V_k \hat{w}\right\| \\
&=\min _{\|w\|=1}\left\|\tilde s^2 Q_A^{\mathrm{T}} Q_A V_k w-\tilde c^2 Q_i^{\mathrm{T}} Q_L V_k w\right\| \\&=\min _{\|w\|=1}\left\|\tilde s^2\left(V_k B_k^{\mathrm{T}} B_k w+\alpha_{k+1} \beta_{k+1} v_{k+1} e_k^{\mathrm{T}} w\right)-\tilde c^2\left(V_k \bar{B}_k^{\mathrm{T}} \bar{B}_k w-\hat{\alpha}_k \hat{\beta}_k v_{k+1} e_k^{\mathrm{T}} w\right)\right\| \\
&=\min _{\|w\|=1}\left\|\left(\tilde s^2 V_k B_k^{\mathrm{T}} B_k-\tilde c^2 V_k \bar{B}_k^{\mathrm{T}} \bar{B}_k+\alpha_{k+1} \beta_{k+1} v_{k+1} e_k^{\mathrm{T}}\right) w\right\| \\
&=\min _{\|w\|=1}\left\|\left(\begin{array}{c}
\tilde s^2 B_k^{\mathrm{T}} B_k-\tilde c^2 \bar{B}_k^{\mathrm{T}} \bar{B}_k \\
\alpha_{k+1} \beta_{k+1} e_k^{\mathrm{T}}
\end{array}\right) w\right\| \\
&=\sigma_{\min }\left(\begin{array}{c}
\tilde s^2 B_k^{\mathrm{T}} B_k-\tilde c^2 \bar{B}_k^{\mathrm{T}} \bar{B}_k \\
\alpha_{k+1} \beta_{k+1} e_k^{\mathrm{T}}
\end{array}\right).
\end{aligned}
$$
This completes the proof.
\end{proof}

This theorem indicates that we can obtain $\hat w$ efficiently and accurately by
computing the SVD of the small $(k+1)\times k$ tridiagonal matrix $\left(\begin{array}{c}
\tilde s^2 B_k^{\mathrm{T}} B_k-\tilde c^2 \bar{B}_k^{\mathrm{T}} \bar{B}_k \\
\alpha_{k+1} \beta_{k+1} e_k^{\mathrm{T}}
\end{array}\right)$.

Once $\hat w$ and $\hat x$ are available, we compute $\hat p^A$ and $\hat p^L$ by (\ref{definition of refined left singular vectors}). Recall that $\tilde{c}=\|A\tilde{x}\|$ and $\tilde{s}=\|L\tilde{x}\|$.
From \eqref{definition of refined left singular vectors} we can
obtain a more accurate appoximate generalized singular value $\hat{c}/\hat{s}$,
called the refined Ritz value, with
\begin{equation}
\label{definition of new approximate generalized singular values}
\hat c=\|A\hat x\|=\|B_k\hat w\|, \quad \hat s=\|L\hat x\|=\|\bar B_k\hat w\|,
\end{equation}
which satisfy $\hat c^2+\hat s^2=1$ and are the solution of the residual minimization problem
$$
\min_{\eta^2+\gamma^2=1} \left\|\eta^2 Q_A^{\mathrm{T}} Q_A V_k\hat w-\gamma^2 Q_L^{\mathrm{T}} Q_L V_k \hat w\right\|.
$$
Such refined Ritz value guarantees that the upper two parts of the corresponding residual of
$\{\hat{c},\hat{s},\hat{x},\hat{p}^A,\hat{p}^L\}$ as defined by \eqref{residual} equal $0$, i.e., $A\hat x=\hat
c\hat p^A$ and $L\hat x=\hat s\hat p^L$. We call $\{\hat c,\hat s,\hat x,\hat p^A,\hat p^L\}$ a refined Ritz
approximation to the desired GSVD component $\{c,s,x,p^A,p^L\}$ of $\{A,L\}$.

Below we establish a compact upper bound on the norm of
residual $r(\hat c,\hat s,\hat{x},\hat p^A,\hat p^L)$, which is used to judge the convergence of $\{\hat
c,\hat s,\hat{x},\hat p^A,\hat p^L\}$ very cheaply without explicitly computing
the right and left Ritz vectors at expensive cost.

\begin{theorem}
\label{Th: residual bound of a refined generalized singular component}
The norm of residual of a refined Ritz approximation $\{\hat c,\hat s,\hat x,\hat p^A,\hat p^L\}$
defined by \eqref{residual} satisfies
\begin{equation}
\label{refined residual bound}
\left\|r\left(\hat c, \hat s, \hat x, \hat p^A, \hat p^L\right)\right\|\leqslant\frac{\|R\|}{\hat c\hat s}\sqrt{\left\|\hat s^2 B_k^{\mathrm{T}} B_k \hat w-\hat c^2 \bar{B}_k^{\mathrm{T}} \bar{B}_k \hat w\right\|^2+\alpha_{k+1}^2 \beta_{k+1}^2\left(e_k^{\rm T} \hat w\right)^2}.
\end{equation}
\end{theorem}

\begin{proof}
Since $A\hat x=\hat c\hat p^A$ and $L\hat x=\hat s\hat p^L$ in the residual $r\left(\hat c, \hat s, \hat x, \hat p^A, \hat p^L\right)$, we have
$$
\left\|r\left(\hat c, \hat s, \hat x, \hat p^A, \hat p^L\right)\right\|
= \|\hat s A^{\mathrm{T}}
\hat p^A-\hat c L^{\mathrm{T}}p^L\| =\|\hat s A^{\mathrm{T}}U_{k+1}
\hat q^A-\hat c L^{\mathrm{T}}\hat{U}_k q^L\|.
$$
From \eqref{Lanczos bidiagonalization}, $h_{k+1}=R^{-1}v_{k+1}$, and $H_k=R^{-1} V_k$, we obtain
\begin{eqnarray*}
A^{\mathrm T}U_{k+1} & =&R^{\mathrm T}Q_A^{\mathrm T}U_{k+1} 
=R^{\mathrm T}\left(V_k B_k^{\mathrm T}+\alpha_{k+1} v_{k+1} e_{k+1}^{\mathrm T}\right)\nonumber \\
& =&R^{\mathrm{T}}R\left(H_k B_k^{\mathrm{T}}+\alpha_{k+1} h_{k+1} e_{k+1}^{\mathrm{T}}\right), \\
L^{\mathrm T}\widehat U_{k} & =&R^{\mathrm T}Q_L^{\mathrm T}\widehat U_{k}
=R^{\mathrm T}\left(\widehat{V}_k \widehat{B}_k^{\mathrm T}+\hat{\beta}_k \hat{v}_{k+1} e_k^{\mathrm T}\right) \nonumber \\
& =&R^{\mathrm{T}}R\left(H_k \bar B_k^{\mathrm{T}}+\bar\beta_{k} h_{k+1}
e_{k}^{\mathrm{T}}\right).\nonumber
\end{eqnarray*}
Thus,
$$
\begin{aligned}
&\left\|r\left(\hat c, \hat s, \hat x, \hat p^A, \hat p^L\right)\right\| \\
&=\left\|R^{\mathrm{T}}\left(\hat s Q_A^{\mathrm{T}} U_{k+1} \hat q^A-\hat c Q_L^{\mathrm{T}} \widehat{U}_k \hat q^L\right)\right\| \\
&=\left\|R^{\mathrm{T}}\left[\frac{\hat s}{\hat c}\left(V_k B_k^{\mathrm{T}} B_k \hat w+\alpha_{k+1} v_{k+1} e_{k+1}^{\mathrm{T}} B_k \hat w\right)-\frac{\hat c}{\hat s}\left(V_k \bar{B}_k^{\mathrm{T}} \bar{B}_k \hat w+\hat{\beta}_k \hat{v}_{k+1} e_k^{\mathrm{T}} \bar{B}_k \hat w\right)\right]\right\| \\
&=\left\|R^{\mathrm{T}}\left[\frac{\hat s}{\hat c} V_k B_k^{\mathrm{T}} B_k \hat w-\frac{\hat c}{\hat s} V_k \bar{B}_k^{\mathrm{T}} \bar{B}_k \hat w+\left(\frac{\hat s}{\hat c}+\frac{\hat c}{\hat s}\right) \alpha_{k+1} \beta_{k+1} v_{k+1} e_k^{\mathrm{T}} \hat w\right]\right\| \\
&\leqslant\frac{\|R\|}{\hat c\hat s}\sqrt{\left\|\hat s^2 B_k^{\mathrm{T}} B_k \hat w-\hat c^2 \bar{B}_k^{\mathrm{T}} \bar{B}_k \hat w\right\|^2+\alpha_{k+1}^2 \beta_{k+1}^2\left(e_k^{\rm T} \hat w\right)^2},
\end{aligned}
$$
where we used $e_{k+1}^{\rm T}B_k\hat w=\beta_{k+1}e_k^{\rm T}\hat w$,
which completes the proof.
\end{proof}

From \eqref{refined residual bound} and \cite{IRJBD},
we get the relative residual norm
\begin{equation}\label{relativeres}
\frac{1}{\|R\|}\left\|r\left(\hat c, \hat s, \hat x, \hat p^A, \hat p^L\right)\right\|\leqslant\frac{1}{\hat c\hat s}\sqrt{\left\|\hat s^2 B_k^{\mathrm{T}} B_k \hat w-\hat c^2 \bar{B}_k^{\mathrm{T}} \bar{B}_k \hat w\right\|^2+\alpha_{k+1}^2 \beta_{k+1}^2\left(e_k^{\rm T} \hat w\right)^2}.
\end{equation}
We claim the refined Ritz approximation
$\{\hat c,\hat s,\hat x,\hat p^A,\hat p^L\}$ converged if
\begin{equation}
\label{stopping criterion for refined}
\frac{1}{\hat c\hat s}\sqrt{\left\|\hat s^2 B_k^{\mathrm{T}} B_k \hat w-\hat c^2 \bar{B}_k^{\mathrm{T}} \bar{B}_k \hat w\right\|^2+\alpha_{k+1}^2 \beta_{k+1}^2\left(e_k^{\rm T} \hat w\right)^2}\leqslant tol,
\end{equation}
where $tol$ is a user-prescribed stopping tolerance. With this stopping criterion
we do not need to explicitly compute right and left
refined Ritz vectors $\hat x$, $\hat p^A$ and $\hat p^L$ before they converge, thereby saving considerable
computational cost without solving the large linear system \eqref{calculate x_i by lsqr} for $\hat{x}$, in
which $\tilde{w}_i$ is replaced by $\hat{w}$.
In \cite{IRJBD}, the authors have established a similar bound for the residual norm of the Ritz approximation and used it to design an efficient stopping criterion without explicitly computing the left and right Ritz vectors until convergence.

\begin{remark}\label{finiteprecision}
In finite precision arithmetic, 
the authors \cite{IRJBD} have
shown that the reliability of the stopping criterion \eqref{stopping criterion for refined} depends
on the size of $\|\underline{B}_k^{-1}\|\|\widehat B_k^{-1}\|$,
where $\underline{B}_k$ is the $k\times k$ leading submatrix of $B_k$. For the RJBD method,
there is an extra term $O(\|\underline{B}_k^{-1}\|\|\widehat B_k^{-1}\|\epsilon_{\rm mach})$
in the right-hand side of \eqref{relativeres}, where $\epsilon_{\rm mach}$ is machine precision.
When $\|\underline{B}_k^{-1}\|\|\widehat
B_k^{-1}\|$ is large, the true residual norms may be larger than the residual norm bound in \eqref{refined
residual bound}, causing that the stopping criterion (\ref{stopping criterion for refined}) is not reliable.
It is proved in \cite{IRJBD,Jia-JBD-finite-precision} that $\|\underline{B}_k^{-1}\|\|\widehat
B_k^{-1}\|$ is uniformly bounded for flat or square $A$ of row full rank and tall or square $L$ of column
full rank
but  tall $A$'s and flat $L$'s may lead to
large $\|\underline{B}_k^{-1}\|$ and $\|\widehat B_k^{-1}\|$, respectively.
Therefore, the stopping criterion is reliable for flat or square $A$ of full row rank
and tall or square $L$ of full column rank but
may not be reliable for $A$ tall and $L$ flat.
\end{remark}

\begin{remark}\label{lsaccuracy}
If the LS problem \eqref{least-squares} is approximately solved by LSQR with the stopping
criterion that requires that the relative residual norm of normal equation of \eqref{least-squares}
drop below a user-prescribed tolerance $lsqrtol$, there is an extra
term $2\sqrt{k+1}\kappa([A;L])\cdot lsqrtol$ in the right-hand side
of \eqref{relativeres},
where $\kappa([A;L])$ is the condition number of the stacked matrix $(A^{\rm T},L^{\rm T})^{\rm T}$.
Together with \cref{finiteprecision}, this indicates
that, supposing that the size of $\|\underline{B}_k^{-1}\|\|\widehat
B_k^{-1}\|$ is modest, the ultimately attainable residual norm by RJBD is
controlled by the stopping tolerance
$lsqrtol$ of LSQR and the condition number $\kappa([A;L])$; see Theorem 3.12 of \cite{IRJBD} and a remark
on it for more details.
\end{remark}

Jia and Stewart \cite{JiaandStewart2001} establish an error bound on the refined Ritz vectors for eigenvalue problems. Exploiting $Q_A^{\mathrm T}Q_A+Q_L^{\mathrm T}Q_L=I_n$ and $\tilde c^2+\tilde s^2=1$, it is straightforward to verify that (\ref{definition of refined right singular vector}) is equivalent to
$$
\left\|HV_k\hat w-\left(\tilde c^2-\tilde s^2\right)V_k\hat w\right\|=\min_{\left\|w\right\|=1}\left\|HV_kw-\left(\tilde c^2-\tilde s^2\right)V_kw\right\|.
$$
Thus, a direct application of Theorem 7.1 in \cite{JiaandStewart2001} leads to the following result.

\begin{theorem}
\label{Th: error of refined right generalized singular vector}
Let $H_{\perp}$ be defined by \eqref{definition of Y_perp}, and assume that
\begin{equation}
\label{condition of Th: error of refined right generalized singular vector}
\operatorname{sep}\left(\tilde c^2-\tilde s^2,H_\perp\right)\geqslant\operatorname{sep}\left(c^2-s^2,H_\perp\right)-\left|\tilde c^2-\tilde s^2-\left(c^2-s^2\right)\right|>0.
\end{equation}
Then
\begin{equation}
\label{tmp202509162215}
\sin\angle\left(\hat x,x\right)_M\leqslant\frac{\left\|H-\left(\tilde c^2-\tilde s^2\right)I\right\|\epsilon+\left|\tilde c^2-\tilde s^2-\left(c^2-s^2\right)\right|}{\sqrt{1-\epsilon^2}\left(\operatorname{sep}\left(c^2-s^2,H_\perp\right)-\left|\tilde c^2-\tilde s^2-\left(c^2-s^2\right)\right|\right)}.
\end{equation}
\end{theorem}

\Cref{th: eigenvalue of BkTBk-barBkTbarBk} has shown that $\left|\tilde c^2-\tilde
s^2-\left(c^2-s^2\right)\right|\rightarrow 0$ as $\epsilon\rightarrow0$. Since ${\rm sep}(c^2-s^2,
H_{\perp})$ is a positive constant, the denominator in (\ref{tmp202509162215}) is uniformly positive from
below as $\epsilon\rightarrow 0$. As a result, the condition (\ref{condition of Th: error of refined right
generalized singular vector}) is guaranteed, and \cref{Th: error of refined right generalized singular vector}
shows that we must have $\hat x\rightarrow x$ unconditionally as $\epsilon\rightarrow 0$,
which is unlike the right Ritz vector $\tilde x$ whose convergence is {\em conditional} and may fail
when $\epsilon\rightarrow 0$.

Regarding the convergence of the left refined Ritz vectors, for a nontrivial generalized singular value $c/s$, by making use of  $A\hat x=\hat
c\hat p^A$ and $L\hat x=\hat s\hat p^L$ as in the JBD method, we can obtain an analog of
\eqref{error of left generalized singular vectors}:
$$
\begin{aligned}
& \sin \angle\left(\hat p^A, p^A\right) \leqslant \frac{\left\|Q_A\right\|}{\hat c} \sin\angle\left(\hat x,x\right)_M, \\
& \sin \angle\left(\hat p^L, p^L\right) \leqslant \frac{\left\|Q_L\right\|}{\hat s} \sin\angle\left(\hat x,x\right)_M,
\end{aligned}
$$
which indicate that $\hat p^A$ and $\hat p^L$ converge since $\hat x$ converges.
Unfortunately, for zero or infinite generalized singular vectors, i.e., $c=0$ or $s=0$,
they fail to converge since $\hat c$ or $\hat
s$ tends to zero. This is in accordance with Theorem 3.1 in \cite{IRJBD}, because the deviation of the
left generalized singular vector $p^A$ or $p^L$ from the subspace $\operatorname{span}(U_{k+1})$ or
$\operatorname{span}(\widehat U_k)$ is generally a positive constant and does {\em not} tend to zero.

\section{Implicit restarting of the RJBD method and selection of the shifts}\label{sec:5}

As the subspace dimension $k$ increases, the JBD and RJBD methods become impractical due to excessive
computational cost and/or storage. Therefore, one must limit the maximum subspace dimension
$k_{\max}$ to perform the
methods. If they do not yet converge, restarting is necessary. The basic idea of restart is to choose a new
initial vector based on the information currently available and construct better left and right subspaces,
from which the JBD and RJBD methods obtain better approximations to the desired GSVD components.
One proceeds in such a way until convergence. The authors in \cite{IRJBD} have nontrivially adapted
the implicit restart
technique \cite{implicit-restarted-Arnoldi} to the JBD process, and developed an implicitly restarted JBD
(IRJBD) algorithm. In an implicitly restarted algorithm,
one must select certain shifts properly, which is crucial to the success and overall
efficiency of the algorithm, as we have shown in \cite{IRJBD}.

We briefly review the machansim of implicit restarting the JBD method \cite{IRJBD}, which applies
to the RJBD method as well. Suppose that the $l$ largest or smallest GSVD components are of interest. Given
$k-l$ pairs of nonnegative
shifts $\{\lambda_i,\mu_i\}_{i=1}^{k-l}$ satisfying $\lambda_i^2+\mu_i^2=1$, implicit
restarting implements $k-l$ step implicit QR iterations on $B_k$ and $\bar{B}_k$ with the $k-l$
shifts, and obtains an updated new $l$-step JBD process
\begin{equation}
\label{l-step restarted JBD}
    \begin{aligned}
    & \left(I_m, 0_{m,p}\right) V_l^{+\prime}=U_{l+1}^+ B_l^+, \\
    & Q Q^{\mathrm T}\left(\begin{array}{c}
    U_{l+1}^+ \\
    0_{p,l+1}
    \end{array}\right)=V_l^{+\prime} B_l^{+\mathrm T}+r_l^{+\prime} e_{l+1}^{\mathrm T}, \\
    & \left(0_{p,m}, I_p\right) V_l^{+\prime} D_l=\widehat{U}_l^+ \bar{B}_l^+
    \end{aligned}
\end{equation}
with $D_l=\operatorname{diag}(1,-1,\dots,(-1)^{l-1})\in\mathbb{R}^{l\times l}$,
and
\begin{equation}
\begin{aligned}
&AH_l^+=U_{l+1}^+B_l^+,\quad LH_l^+=\widehat U_l^+\bar B_l^+,\\
&B_l^{+\mathrm T}B_l^++\bar B_l^{+\mathrm T}\bar B_l^+=I_l,
\end{aligned}
\end{equation}
where the updated starting vectors
\begin{equation}
\label{start vectors after restart}
\begin{aligned}
&\tau u_1^+=\left(Q_A Q_A^{\mathrm T}-\lambda_1^2I_m\right)\cdots\left(Q_AQ_A^{\mathrm T}-\lambda_{k-l}^2I_m\right)u_1,\\
&\rho v_1^+=\left(Q_A^{\mathrm T}Q_A-\lambda_1^2I_n\right)\cdots\left(Q_A^{\mathrm T}Q_A-\lambda_{k-l}^2I_n\right)v_1\\
&\ \ \ \ \ =\left(Q_L^{\mathrm T}Q_L-\mu_1^2I_n\right)\cdots\left(Q_L^{\mathrm T}Q_L-\mu_{k-l}^2I_n\right)v_1,
\end{aligned}
\end{equation}
where $\tau$ and $\rho$ are normalizing factors. One then extends it to the $k$-step JBD process
in a standard way. For efficient and reliable implementations in finite precision arithmetic,
we refer to \cite{IRJBD} for details.

The selection of shifts is crucial for making a JBD-process-based algorithm efficient. On the basis of work in
\cite{Jia-poly-refined,Jia-IRRBL}, the authors \cite{IRJBD} have shown that the closer the shifts are to those
unwanted generalized singular values, the richer information on the desired generalized singular
vectors is contained in the resulting restarted subspaces. Based on this result,
the authors take the $k-l$ unwanted Ritz values as shifts, called the exact shifts,
which are the best available approximations obtained by the JBD method to some of those unwanted generalized singular values.


To get more insight into the exact shifts, let us make the orthogonal direct sum decompositions
\begin{equation}
\label{direct sum decompositions by Ritz vectors}
\begin{aligned}
\text{span}(U_{k+1})&=\text{span}(\tilde p_1^A,\dots,\tilde p_l^A)\oplus\text{span}(\tilde p_1^A,\dots,\tilde p_l^A)^{\perp},\\
\text{span}(\widehat U_k)&=\text{span}(\tilde p_1^L,\dots,\tilde p_l^L)\oplus\text{span}(\tilde p_1^L,\dots,\tilde p_l^L)^{\perp},\\
\text{span}(V_k)&=\text{span}(R\tilde x_1,\dots,R\tilde x_l)\oplus\text{span}(R\tilde x_1,\dots,R\tilde x_l)^{\perp}.
\end{aligned}
\end{equation}
It is straightforward to verify the following theorem, which is presented only for the computation of
the $l$ largest GSVD components, and the computation of the $l$ smallest GSVD components is similar
except the subscript changes in the quantities involved.

\begin{theorem}
In IRJBD, the $l$ largest Ritz values are the generalized singular values of
 $\{Q_A,Q_L\}$ with respect to the left
subspaces $\text{span}(\tilde p_1^A,\dots,\tilde p_l^A)$ and $\text{span}(\tilde p_1^L,\dots,\tilde p_l^L)$ and the right subspace $\text{span}(R\tilde x_1,\dots,R\tilde x_l)$, and
the exact shifts used are the generalized singular values of $\{Q_A,Q_L\}$ with respect to the left
subspaces $\text{span}(\tilde p_1^A,\dots,\tilde p_l^A)^{\perp}$ and $\text{span}(\tilde p_1^L,\dots,\tilde p_l^L)^{\perp}$, and the right subspace $\text{span}(R\tilde x_1,\dots,R\tilde x_l)^{\perp}$.
\end{theorem}

Motivated by this theorem and the decompositions in \eqref{direct sum decompositions by Ritz vectors},
we are able to find new shifts for use within the IRRJBD
algorithm, such that they are more accurate approximations than the exact shifts to some of the undesired
generalized singular values of $\{A,L\}$. For RJBD, make the orthogonal direct sum
decompositions similar to those in \eqref{direct sum decompositions by Ritz vectors}:
\begin{equation}
\label{direct sum decompositions by refined Ritz vectors}
\begin{aligned}
\text{span}(U_{k+1})&=\text{span}(\hat p_1^A,\dots,\hat p_l^A)\oplus\text{span}(\hat p_1^A,\dots,\hat p_l^A)^{\perp},\\
\text{span}(\widehat U_k)&=\text{span}(\hat p_1^L,\dots,\hat p_l^L)\oplus\text{span}(\hat p_1^L,\dots,\hat p_l^L)^{\perp},\\
\text{span}(V_k)&=\text{span}(R\hat x_1,\dots,R\hat x_l)\oplus\text{span}(R\hat x_1,\dots,R\hat x_l)^{\perp}.
\end{aligned}
\end{equation}
Since $\text{span}(\hat p_1^A,\dots,\hat p_l^A)$, $\text{span}(\hat p_1^L,\dots,\hat p_l^L)$ and
$\text{span}(R\hat x_1,\dots,R\hat x_l)$ generally contain more accurate
approximations to the desired $l$ left
and right generalized singular vectors than $\text{span}(\tilde p_1^A,\dots,\tilde p_l^A)$,
$\text{span}(\tilde p_1^L,\dots,\tilde p_l^L)$ and $\text{span}(R\tilde x_1,\dots,R\tilde x_l)$ do, their
orthogonal complements $\text{span}(\hat p_1^A,\dots,\hat p_l^A)^{\perp}$, $\text{span}(\hat
p_1^L,\dots,\hat p_l^L)^{\perp}$ and $\text{span}(R\hat x_1,\dots,R\hat x_l)^{\perp}$ contain more
accurate approximations to
some $k-l$ ones among the undesired generalized singular vectors than $\text{span}(\tilde
p_1^A,\dots,\tilde
p_l^A)^{\perp}$, $\text{span}(\tilde p_1^L,\dots,\tilde p_l^L)^{\perp}$ and $\text{span}(R\tilde
x_1,\dots,R\tilde x_l)^{\perp}$ do, respectively. As a result, the generalized singular values of
$\{Q_A,Q_L\}$ with respect to the left subspaces $\text{span}(\hat p_1^A,\dots,\hat p_l^A)^{\perp}$ and
$\text{span}(\hat p_1^L,\dots,\hat p_l^L)^{\perp}$, and the right subspace $\text{span}(R\hat x_1,\dots,R\hat
x_l)^{\perp}$ are generally more accurate approximations to $k-l$ undesired generalized singular values than the exact shifts. Therefore, they are generally better than
the exact shifts and
the resulting restarted subspaces contain more accurate approximations to the desired left and
right generalized singular vectors.

Since these new shifts originate from the RJBD method, we call them the
refined shifts, and use them within the IRRJBD algorithm. Similarly, in the sense of (\ref{direct sum
decompositions by refined Ritz vectors}), the refined shifts are the best  possible approximations to some of
the undesired generalized singular values and thus the best shifts for use within IRRJBD.
Consequently, it is expected that IRRJBD uses fewer or at least no more restarts to
converge than IRJBD for the same $k, l$ and starting
vector $u_1$. Based on the above arguments, we present a formal and clear definition of the refined shifts.

\begin{definition}
\label{Def: refined shifts}
Let $\widehat P_l^A=(\hat p_1^A,\dots,\hat p_l^A)$, $\widehat P_l^L=(\hat p_1^L,\dots,\hat p_l^L)$ and
$\widehat X_l=(\hat x_1,\dots,\hat x_l)$, and $\widehat P_l^{A\perp}$, $\widehat P_l^{L\perp}$ and $R\widehat
X_l^{\perp}$ be orthonormal matrices whose columns form orthonormal bases
of ${\rm span}(\hat p_1^A,\dots,\hat p_l^A)^{\perp}$, ${\rm span}(\hat p_1^L,\dots,\hat p_l^L)^{\perp}$
and $\text{span}(R\hat x_1,\dots,R\hat x_l)^{\perp}$, respectively.
The refined shifts are the generalized singular
values of $\{(\widehat P_l^{A\perp})^{\mathrm T}Q_AR\widehat X_l^{\perp},(\widehat P_l^{L\perp})^{\mathrm
T}Q_LR\widehat X_l^{\perp}\}$.
\end{definition}

Now we show how to compute the refined shifts efficiently and accurately.
Notice from \eqref{definition of refined right singular vector} and
\eqref{definition of refined left singular vectors} that $\hat p_i^A=U_{k+1}\hat q_i^A$, $\hat p_i^L=\widehat U_k\hat q_i^L$, and $R\hat x_i=V_k\hat w_i$.
Write the matrices $\widehat Q_l^A=(\hat q_1^A,\dots,\hat q_l^A)\in \mathbb{R}^{(k+1)\times l}$,
$\widehat Q_l^L=(\hat q_1^L,\dots,\hat q_l^L)\in \mathbb{R}^{k\times l}$, and
$\widehat W_l=(\hat w_1,\dots,\hat w_l)\in \mathbb{R}^{k\times l}$. Compute their full QR factorizations
\begin{equation}
\label{full OR of low-dimensional refined Ritz vectors}
\widehat W_l=Q_{\widehat W_l}R_{\widehat W_l}, \quad\widehat Q_l^A=Q_{\widehat Q_l^A}R_{\widehat Q_l^A},\quad \widehat Q_l^L=Q_{\widehat Q_l^L}R_{\widehat Q_l^L},
\end{equation}
and partition
\begin{equation}
\label{partition Q-factors of full QR}
Q_{\widehat W_l}=(\text{orth}(\widehat W_l),\,\widehat W_l^{\perp}), \,Q_{\widehat Q_l^A}=(\text{orth}(\widehat Q_l^A),\,\widehat Q_l^{A\perp}), \,Q_{\widehat Q_l^L}=(\text{orth}(\widehat Q_l^L),\,\widehat Q_l^{L\perp}),
\end{equation}
where $\text{orth}(\cdot)$ denotes the Q-factor in the thin QR factorization of
a matrix. Then
\begin{equation}
\begin{aligned}
&\text{span}\left\{\widehat P_l^{A\perp}\right\}=\text{span}\left\{U_{k+1}\widehat Q_l^{A\perp}\right\}, \\
&\text{span}\left\{\widehat P_l^{L\perp}\right\}=\text{span}\left\{\widehat U_{k}\widehat Q_l^{L\perp}\right\}, \\
&\text{span}\left\{R\widehat X_l^{\perp}\right\}=\text{span}\left\{V_k\widehat W_l^{\perp}\right\}.
\end{aligned}
\end{equation}
From \eqref{Lanczos bidiagonalization},
we have $U_{k+1}^{\rm T}Q_AV_k=B_k$ and $\widehat{U}_k^{\rm T}Q_LV_k=\bar{B}_k$. Thus,
by \cref{Def: refined shifts}, the refined shifts are the generalized singular values
of the $(k-l)\times (k-l)$ matrix pair $\{(\widehat Q_l^{A\perp})^{\mathrm T}B_k\widehat W_l^{\perp},(\widehat
Q_l^{L\perp})^{\mathrm T}\bar B_k\widehat W_l^{\perp}\}$, and can be computed using only $O((k-l)^3)$
flops. We summarize the procedure as \cref{alg:shifts}.

\begin{algorithm}[H]
	\renewcommand{\algorithmicrequire}{\textbf{Input:}}
	\renewcommand{\algorithmicensure}{\textbf{Output:}}
	\caption{The computation of refined shifts}
	\label{alg:shifts}
	\begin{algorithmic}[1]
        \REQUIRE $B_k$, $\bar B_k$, $\{\hat w_i\}_{i=1}^l$, $\{\hat q_i^A\}_{i=1}^l$, and $\{\hat q_i^L\}_{i=1}^l$.
        \ENSURE $k-l$ refined shifts $\{\hat\lambda_i,\hat\mu_i\}_{i=1}^{k-l}$ satisfying $\hat\lambda_i^2+\hat\mu_i^2=1$.
        \STATE Form $\widehat W_l=(\hat w_1,\dots,\hat w_l)$, $\widehat Q_l^A=(\hat q_1^A,\dots,\hat q_l^A)$ and $\widehat Q_l^L=(\hat q_1^L,\dots,\hat q_l^L)$.
        \STATE Compute the full QR factorizations (\ref{full OR of low-dimensional refined Ritz vectors}) of $\widehat W_l$, $\widehat Q_l^A$ and $\widehat Q_l^L$ to get $Q_{\widehat W_l}$, $Q_{\widehat Q_l^A}$ and $Q_{\widehat Q_l^L}$.
        \STATE Partition $Q_{\widehat W_l}$, $Q_{\widehat Q_l^A}$ and $Q_{\widehat Q_l^L}$ as (\ref{partition Q-factors of full QR}) to obtain $\widehat W_l^{\perp}$, $\widehat Q_l^{A\perp}$ and $\widehat Q_l^{L\perp}$.
        \STATE Form $(\widehat Q_l^{A\perp})^{\mathrm T}B_k\widehat W_l^{\perp}$ and $(\widehat Q_l^{L\perp})^{\mathrm T}\bar B_k\widehat W_l^{\perp}$, and compute the generalized singular values $\{\hat\lambda_i,\hat\mu_i\}_{i=1}^{k-l}$ of $\{(\widehat Q_l^{A\perp})^{\mathrm T}B_k\widehat W_l^{\perp},\,(\widehat Q_l^{L\perp})^{\mathrm T}\bar B_k\widehat W_l^{\perp}\}$.
	\end{algorithmic}
\end{algorithm}

\section{Summary of IRRJBD}\label{sec:6}

We still use the ``+3" strategy in implicitly restarted algorithms   \cite{IRJBD,Jia-poly-refined,implicit-restarted-Arnoldi}; that is, the number of shifts is $k-(l+3)$ when
the $l$ eigenpairs, SVD or GSVD components are computed. For IRRJBD, the default parameter $adjust=3$
means that $k-(l+3)$ refined shifts are used for implicit restart. Algorithm \ref{alg:IRRJBD} describes a
complete IRRJBD algorithm, and Table \ref{tab:parameters} lists the parameters used by IRRJBD and their
default values.

\begin{algorithm}[ht]
	\renewcommand{\algorithmicrequire}{\textbf{Input:}}
	\renewcommand{\algorithmicensure}{\textbf{Output:}}
	\caption{The IRRJBD algorithm}
	\label{alg:IRRJBD}
	\begin{algorithmic}[1]
        \REQUIRE Matrices $A \in \mathbb{R}^{m \times n}$, $L \in \mathbb{R}^{p \times n}$, vector $u_1 \in \mathbb{R}^m$ with $\|u_1\|=1$, the number $l=|t|$ of the desired extreme GSVD components with $t>0$ or $t<0$ indicating that the largest or
        smallest GSVD components are required, $k_{\max}$, $adjust$, the maximum restarts $maxit$,
        and the stopping tolerance $tol$ in (\ref{stopping criterion for refined}).
        \ENSURE $l$ converged $\{\hat c_i,\,\hat s_i,\,\hat p_i^A,\,\hat p_i^L,\,\hat x_i\}_{i=1}^l$.
        \STATE Do the $l$-step JBD process of $\{A,\,L\}$.
        \WHILE{the $l$ refined Ritz approximations not converged or restarts $< maxit$ }
        \label{determining convergence in alg implementation}
            \IF{the basis size $<k_{\max}$}
                \STATE Do one step of the JBD process.
            \ELSE
                \STATE Compute the refined shifts according to \Cref{alg:shifts}.
                \STATE Implicitly restart the JBD process using the refined shifts.
            \ENDIF
        \ENDWHILE
        \STATE Compute the $l$ converged right refined Ritz vectors $\hat x_i = R^{-1} V_k \hat w_i$
        by solving the consistent linear systems \eqref{calculate x_i by lsqr},
        in which $\tilde w_i$ is
        replaced with $\hat w_i$.
        \STATE Compute the $l$ converged left refined Ritz vectors $\hat p_i^A=U_{k+1}\hat q_i^A$ and $\hat p_i^L=\widehat U_k\hat q_i^L$, where $q_i^A$ and $q_i^L$ are defined in
        \eqref{definition of refined left singular vectors}.
        \STATE Compute the $l$ refined Ritz values $\{\hat c_i, \,\hat s_i\}$ as
        \eqref{definition of new approximate generalized singular values}.
	\end{algorithmic}
\end{algorithm}

Here are some implementation details of the algorithm. By default, the convergence of
line \ref{determining convergence in alg implementation} in \Cref{alg:IRRJBD} is judged by
\eqref{stopping criterion for refined} with the stopping tolerance $tol=10^{-8}$.
We solve the LS problems in \eqref{least-squares}
and consistent linear systems in \eqref{calculate x_i by lsqr} by the {\sc Matlab} built-in function
{\sf lsqr} with the stopping tolerance $lsqrtol=tol/100=10^{-10}$, and
set the maximum iteration steps of {\sf lsqr} to $n$.
We perform full reorthogonalization on $V_k^{\prime},\,U_{k+1}$ and $\widehat
U_k$ to ensure that they are numerically orthonormal to the
working precision. The starting vector $u_1$ is randomly
generated by the standard normal distribution function {\sf randn} in {\sc Matlab}
and then normalized. If $tol$ is smaller than the default value, we set
$lsqrtol=\max\{10\epsilon_{\rm mach},tol/100\}$.

\begin{table}[ht]
\centering
\caption{Parameters of IRRJBD}
\begin{tabularx}{\textwidth}{>{\raggedright\arraybackslash}p{0.15\textwidth}%
                                    >{\raggedright\arraybackslash}p{0.25\textwidth}%
                                    >{\raggedright\arraybackslash}p{0.5\textwidth}}
\hline
Parameters & Default values & Description \\ \hline
$t$ & 5 & $|t|=l$ is the number of desired GSVD components,  where $t>0$ or $t<0$ means that the largest or smallest GSVD components are required. \\ \hline
$k_{\max}$ & $\max\{3l, 20\}$ & Maximum subspace dimension \\ \hline
$adjust$ & $3$ & Integer added to $l$ to speed up convergence \\ \hline
$tol$ & $10^{-8}$ & Stopping tolerance in \eqref{stopping criterion for refined} \\ \hline
$maxit$ & $10000$ & Maximum number of restarts \\ \hline
$lsqrtol$ & $\max\{10\epsilon_{\rm mach},\frac{tol}{100}\}$ & Stopping tolerance of {\sf lsqr} \\ \hline
$lsqrmaxit$ & $n$ & Maximum number of iterations of {\sf lsqr} \\ \hline
$u_1$ & generated by {\sf randn} and then normalized & The unit length starting vector \\ \hline
\end{tabularx}
\label{tab:parameters}
\end{table}

\section{Numerical experiments}\label{sec:7}

In this section, we present numerical experiments to demonstrate the performance of IRRJBD, and
show its superiority to IRJBD in \cite{IRJBD}.
The experiments
were performed on an Intel(R) Core(TM) Ultra 7 265K, 48 GB RAM, and 20 cores using the {\sc Matlab}
R2025b with
the machine precision $\epsilon_{\rm mach}=2.22\times10^{-16}$ under the Windows 11 64-bit system.

We use some matrices from the SuiteSparse Matrix Collection \cite{davis2011university} or their transpose as $A$, and use the following two well-conditioned matrices as $L$:
$$
L_{\text{tall}}=\left(\begin{array}{ccc}
2 & & \\
1 & \ddots &  \\
& \ddots & 2 \\
& & 1
\end{array}\right)\in \mathbb{R}^{(n+1)\times n},\quad
L_{\text{flat}}=\left(\begin{array}{cccc}
2 & 1 & & \\
& \ddots & \ddots & \\
& & 2 & 1
\end{array}\right)\in \mathbb{R}^{(n-1)\times n}.
$$
Table \ref{tab:matrices} lists the test matrices and some of their properties. The first two $A$ are square, the middle two are flat, and the last two are tall. Except $\{\text{lp\_ken\_18}^{\rm T}, L_{\text{tall}}\}$, the values of $\kappa(A)$'s are taken from the \href{https://sparse.tamu.edu/}{SuiteSparse Matrix Collection}. The condition number of lp\_ken\_18 is not reported in the collection.
Using {\sc Matlab} built-in function {\sf svds}, we found
that the computed smallest singular value of lp\_ken\_18$^{\rm T}$ is $7.32\times 10^{-16}$.
Thus, lp\_ken\_18$^{\rm T}$ has a numerically infinite condition number.

\begin{table}[htp!]
\centering
\caption{Test matrices, where $nnz$ is the total number of nonzero entries in $\{A,L\}$, and $\kappa([A;L])$ is the condition number of the stacked matrix
$(A^{\mathrm T},L^{\mathrm T})^{\mathrm T}$, whose modest sizes indicate that the matrix pairs $\{A,L\}$ are regular.
}
\resizebox{0.95\textwidth}{!}{
\begin{tabular}{|c|c|c|c|c|c|c|c|}
\hline
$A$ & L & $m$ & $n$ & $p$ & $nnz$ & $\kappa(A)$ & $\kappa([A;L])$ \\ \hline
nopoly & $L_{\text{tall}}$ & 10774 & 10774 & 10775 & 92390 & 1.21E+16 & 9.9871 \\ \hline
appu & $L_{\text{flat}}$ & 14000 & 14000 & 13999 & 1881102 & 1.71E+02 & 21.4439 \\ \hline
pcb3000 & $L_{\text{tall}}$ & 3960 & 7732 & 7733 & 72943 & 5.67E+02 & 91.4016 \\ \hline
seymourl & $L_{\text{flat}}$ & 4944 & 6316 & 6315 & 51123 & 3.04E+01 & 28.2866 \\ \hline
lp\_ken\_18$^{\rm T}$ & $L_{\text{tall}}$ & 154699 & 105127 & 105128 & 568425 & $+\infty$ & 15.6627 \\ \hline
lp\_dfl001$^{\rm T}$ & $L_{\text{flat}}$ & 12230 & 6071 & 6070 & 47772 & 3.94E+15 & 10.0163 \\ \hline
\end{tabular}
}
\label{tab:matrices}
\end{table}

We compare IRRJBD with IRJBD on the test problems in Table \ref{tab:matrices}.
To be fair, we took the same parameters $adjust=3$,
$tol=10^{-8}$, $lsqrtol=10^{-10}$, and starting vector $u_1$ in IRRJBD and IRJBD
for each test problem. We first investigate if they can
compute the desired GSVD components. Then we compare their efficiency
when they succeed.
Notice that $R$ is unknown in \eqref{relativeres}. In \cite{IRJBD}, the true relative residual norm in the
left-hand side of \eqref{relativeres} is replaced by the
relative residual norm
\begin{equation}\label{actualres}
  \frac{\|r\left(\hat c, \hat s, \hat x, \hat p^A, \hat p^L\right)\|}
  {\sqrt{\|A\|_1\|A\|_{\infty}+\|L\|_1\|L\|_{\infty}}},
\end{equation}
where $\|\cdot\|_1$ and $\|\cdot\|_{\infty}$ denote the 1-norm and $\infty$-norm of a matrix.

First, we present the results for $A$ flat square $A$ and $L$ tall in
Table~\ref{tab: robustness of IRJBD and IRRJBD, square of flat A, tall L}. From
Table~\ref{tab: robustness of IRJBD and IRRJBD, square of flat A, tall L}, we see that both IRJBD
and IRRJBD succeeded for the test problems with the relative residual norm
\eqref{actualres} of the $l$ approximate GSVD components
below the stopping tolerance $tol$ when
$\|\underline{B}_k^{-1}\|\|\widehat B_k^{-1}\|$ had generic sizes;
they failed to converge for large $\|\underline{B}_k^{-1}\|\|\widehat B_k^{-1}\|$.
Specifically, the algorithms computed the largest GSVD components successfully;
but they failed for the computation of the smallest GSVD components with the relative residual
norms larger than $tol$. These phenomena were explained in detail
in \cite{IRJBD} and confirms \cref{finiteprecision} by
noticing that nopoly is numerically rank deficient, in which case the smallest
Ritz value approximated the trivial zero generalized singular value but left Ritz and
refined Ritz vectors for $A$ did not converge.

Furthermore,
we observe from the table that IRRJBD was more efficient than IRJBD in terms of restarts and CPU
time. For $\{\text{nopoly},L_{\text{tall}}\}$, the speedup ratios are
$30\%\sim 50\%$ in half of the cases; for $\{\text{pcb3000},L_{\text{tall}}\}$, the two
algorithms were equally efficient. We also notice that,
for $\{A,L\}=\{\text{nopoly},L_{\text{tall}}\}$, although the two algorithms
did not reach the convergence for subspace dimension 25 after 10000 restarts were used
when computing the 5 and 10 smallest GSVD components,
the $Res$'s obtained by IRRJBD were squares of those by IRJBD, great reductions.

\begin{table}[htp!]
\caption{Effectiveness and efficiency of IRRJBD and IRJBD for square or flat $A$ and tall $L$, where $iter$ is
the number of restarts, $time$ denotes the CPU time, $Res$ is the maximum
relative residual norm
\eqref{actualres} of the
$l$ converged approximations to the desired GSVD components, and the quantities
$B_{\rm IR}$ and $B_{\rm IRR}$ are the values of $\|\underline{B}_k^{-1}\|\|\widehat B_k^{-1}\|$ when the
IRJBD and IRRJBD terminated,
$SI=(iter_{\text{IR}}-iter_{\text{IRR}})/iter_{\text{IR}}$ and
$ST=(time_{\text{IR}}-time_{\text{IRR}})/time_{\text{IR}}$ are the speedup ratios in restarts and CPU time,
respectively, and ``-" indicates that an underlying algorithm did not converge after 10000 restarts were used, in which we do not report the comparison results.  $Res$ in bold denotes
no convergence.
}
\label{tab: robustness of IRJBD and IRRJBD, square of flat A, tall L}
\begin{minipage}{\textwidth}
\centering
\begin{subtable}{\textwidth}
\centering
\caption{$\{A,L\}=\{\text{nopoly},L_{\text{tall}}\}$}
\resizebox{\textwidth}{!}{%
\begin{tabular}{|l|ll|ll|ll|ll|}
\hline
$target$ & \multicolumn{2}{c|}{-5} & \multicolumn{2}{c|}{5} & \multicolumn{2}{c|}{-10} & \multicolumn{2}{c|}{10} \\ \hline
$k_{\max}$ & \multicolumn{1}{l|}{25} & 50 & \multicolumn{1}{l|}{25} & 50 & \multicolumn{1}{l|}{25} & 50 & \multicolumn{1}{l|}{25} & 50 \\ \hline
$iter_{\text{IR}}$ & \multicolumn{1}{l|}{10000} & 7099 & \multicolumn{1}{l|}{535} & 785 & \multicolumn{1}{l|}{10000} & 3600 & \multicolumn{1}{l|}{1514} & 1399 \\ \hline
$iter_{\text{IRR}}$ & \multicolumn{1}{l|}{10000} & 4804 & \multicolumn{1}{l|}{431} & 731 & \multicolumn{1}{l|}{10000} & 2066 & \multicolumn{1}{l|}{1007} & 1322 \\ \hline
$SI$(\%) & \multicolumn{1}{l|}{-} & 32.33 & \multicolumn{1}{l|}{19.44} & 6.88 & \multicolumn{1}{l|}{-} & 42.61 & \multicolumn{1}{l|}{33.49} & 5.50 \\ \hline
$time_{\text{IR}}$ & \multicolumn{1}{l|}{-} & 142795.50 & \multicolumn{1}{l|}{3556.64} & 15342.95 & \multicolumn{1}{l|}{-} & 70585.92 & \multicolumn{1}{l|}{12124.34} & 27627.00 \\ \hline
$time_{\text{IRR}}$ & \multicolumn{1}{l|}{-} & 96299.73 & \multicolumn{1}{l|}{2980.70} & 14582.81 & \multicolumn{1}{l|}{-} & 42959.31 & \multicolumn{1}{l|}{8114.281} & 27202.81 \\ \hline
$ST$(\%) & \multicolumn{1}{l|}{-} & 32.56 & \multicolumn{1}{l|}{16.19} & 4.95 & \multicolumn{1}{l|}{-} & 39.14 & \multicolumn{1}{l|}{33.07} & 1.54 \\ \hline
$Res_{\text{IR}}$ & \multicolumn{1}{l|}{{\bf 1.52E-03}} & 2.28E-09 & \multicolumn{1}{l|}{9.52E-09} & 5.49E-09 & \multicolumn{1}{l|}{{\bf 1.92E-03}} & 2.13E-09 & \multicolumn{1}{l|}{8.31E-09} & 4.88E-09 \\ \hline
$Res_{\text{IRR}}$ & \multicolumn{1}{l|}{{\bf 3.34E-06}} & 8.61E-10 & \multicolumn{1}{l|}{7.78E-09} & 6.12E-09 & \multicolumn{1}{l|}{{\bf 3.90E-06}} & 8.23E-10 & \multicolumn{1}{l|}{6.45E-09} & 5.92E-09 \\ \hline
$B_{\text{IR}}$ & \multicolumn{1}{l|}{3.08E+10} & 6.70E+04 & \multicolumn{1}{l|}{1.04E+02} & 4.11E+02 & \multicolumn{1}{l|}{6.41E+10} & 2.76E+04 & \multicolumn{1}{l|}{1.61E+04} & 1.14E+03 \\ \hline
$B_{\text{IRR}}$ & \multicolumn{1}{l|}{1.57E+11} & 1.80E+05 & \multicolumn{1}{l|}{7.02E+02} & 6.96E+02 & \multicolumn{1}{l|}{1.45E+11} & 2.99E+04 & \multicolumn{1}{l|}{1.55E+02} & 6.52E+02 \\ \hline
\end{tabular}%
}
\end{subtable}
\vspace{0.8em} 
\begin{subtable}{\textwidth}
\centering
\caption{$\{A,L\}=\{\text{pcb3000},L_{\text{tall}}\}$}
\resizebox{\textwidth}{!}{%
\begin{tabular}{|lllllllll|}
\hline
\multicolumn{1}{|l|}{$target$} & \multicolumn{2}{|c|}{-5} & \multicolumn{2}{c|}{5} & \multicolumn{2}{c|}{-10} & \multicolumn{2}{c|}{10} \\ \hline
\multicolumn{1}{|l|}{$k_{\max}$} & \multicolumn{1}{|l|}{25} & \multicolumn{1}{l|}{50} & \multicolumn{1}{l|}{25} & \multicolumn{1}{l|}{50} & \multicolumn{1}{l|}{25} & \multicolumn{1}{l|}{50} & \multicolumn{1}{l|}{25} & 50 \\ \hline
\multicolumn{1}{|l|}{$iter_{\text{IR}}$} & \multicolumn{1}{|l|}{11} & \multicolumn{1}{l|}{4} & \multicolumn{1}{l|}{16} & \multicolumn{1}{l|}{6} & \multicolumn{1}{l|}{19} & \multicolumn{1}{l|}{6} & \multicolumn{1}{l|}{34} & 10 \\ \hline
\multicolumn{1}{|l|}{$iter_{\text{IRR}}$} & \multicolumn{1}{|l|}{11} & \multicolumn{1}{l|}{4} & \multicolumn{1}{l|}{15} & \multicolumn{1}{l|}{6} & \multicolumn{1}{l|}{18} & \multicolumn{1}{l|}{6} & \multicolumn{1}{l|}{35} & 10 \\ \hline
\multicolumn{1}{|l|}{$SI$(\%)} & \multicolumn{1}{|l|}{0.00} & \multicolumn{1}{l|}{0.00} & \multicolumn{1}{l|}{6.25} & \multicolumn{1}{l|}{0.00} & \multicolumn{1}{l|}{5.26} & \multicolumn{1}{l|}{0.00} & \multicolumn{1}{l|}{-2.94} & 0.00 \\ \hline
\multicolumn{1}{|l|}{$time_{\text{IR}}$} & \multicolumn{1}{|l|}{79.02} & \multicolumn{1}{l|}{77.05} & \multicolumn{1}{l|}{112.50} & \multicolumn{1}{l|}{108.16} & \multicolumn{1}{l|}{102.45} & \multicolumn{1}{l|}{103.30} & \multicolumn{1}{l|}{172.92} & 164.72 \\ \hline
\multicolumn{1}{|l|}{$time_{\text{IRR}}$} & \multicolumn{1}{|l|}{80.66} & \multicolumn{1}{l|}{74.53} & \multicolumn{1}{l|}{107.56} & \multicolumn{1}{l|}{109.92} & \multicolumn{1}{l|}{94.30} & \multicolumn{1}{l|}{101.77} & \multicolumn{1}{l|}{179.47} & 165.78 \\ \hline
\multicolumn{1}{|l|}{$ST$(\%)} & \multicolumn{1}{|l|}{-2.08} & \multicolumn{1}{l|}{3.27} & \multicolumn{1}{l|}{4.39} & \multicolumn{1}{l|}{-1.63} & \multicolumn{1}{l|}{7.96} & \multicolumn{1}{l|}{1.48} & \multicolumn{1}{l|}{-3.79} & -0.65 \\ \hline
\multicolumn{1}{|l|}{$Res_{\text{IR}}$} & \multicolumn{1}{|l|}{8.68E-11} & \multicolumn{1}{l|}{7.38E-11} & \multicolumn{1}{l|}{3.02E-09} & \multicolumn{1}{l|}{2.80E-09} & \multicolumn{1}{l|}{1.01E-10} & \multicolumn{1}{l|}{8.59E-11} & \multicolumn{1}{l|}{3.17E-09} & 3.11E-09 \\ \hline
\multicolumn{1}{|l|}{$Res_{\text{IRR}}$} & \multicolumn{1}{|l|}{9.42E-11} & \multicolumn{1}{l|}{7.61E-11} & \multicolumn{1}{l|}{2.79E-09} & \multicolumn{1}{l|}{3.13E-09} & \multicolumn{1}{l|}{1.03E-10} & \multicolumn{1}{l|}{9.28E-11} & \multicolumn{1}{l|}{2.94E-09} & 3.12E-09 \\ \hline
\multicolumn{1}{|l|}{$B_{\text{IR}}$} & \multicolumn{1}{|l|}{2.44E+02} & \multicolumn{1}{l|}{3.27E+02} & \multicolumn{1}{l|}{3.59E+02} & \multicolumn{1}{l|}{3.81E+02} & \multicolumn{1}{l|}{1.84E+02} & \multicolumn{1}{l|}{3.15E+02} & \multicolumn{1}{l|}{1.34E+02} & 3.81E+02 \\ \hline
\multicolumn{1}{|l|}{$B_{\text{IRR}}$} & \multicolumn{1}{|l|}{2.44E+02} & \multicolumn{1}{l|}{3.26E+02} & \multicolumn{1}{l|}{3.51E+02} & \multicolumn{1}{l|}{3.81E+02} & \multicolumn{1}{l|}{1.55E+02} & \multicolumn{1}{l|}{3.15E+02} & \multicolumn{1}{l|}{1.89E+02} & 3.77E+02 \\ \hline
\end{tabular}%
}
\end{subtable}
\end{minipage}
\end{table}

We now conduct a more detailed efficiency comparison of IRRJBD and IRJBD. In
Table~\ref{tab: efficiency of IRJBD and IRRJBD}, both algorithms converged for all the
test problems, i.e., the actual relative residual norms, defined by \eqref{actualres}, of the approximate GSVD components were smaller than
$tol$. Therefore, we do not report $Res$'s
in the table, and only present the numbers of restarts and the
total CPU time.

\begin{table}[htp!]
\caption{More efficiency comparisons of IRRJBD and IRJBD}
\label{tab: efficiency of IRJBD and IRRJBD}
\resizebox{\textwidth}{!}{%
\begin{tabular}{|l|l|l|ll|l|ll|l|}
\hline
\multirow{2}{*}{$(A;L)$} & \multicolumn{1}{c|}{\multirow{2}{*}{$target$}} & \multicolumn{1}{c|}{\multirow{2}{*}{$k_{\max}$}} & \multicolumn{2}{c|}{$iter$} & \multicolumn{1}{c|}{\multirow{2}{*}{$SI$(\%)}} & \multicolumn{2}{c|}{$time$} & \multirow{2}{*}{$ST$(\%)} \\ \cline{4-5} \cline{7-8}
 & \multicolumn{1}{c|}{} & \multicolumn{1}{c|}{} & \multicolumn{1}{l|}{IR} & IRR & \multicolumn{1}{c|}{} & \multicolumn{1}{l|}{IR} & IRR &  \\ \hline
\multirow{6}{*}{\begin{tabular}[c]{@{}l@{}}nopoly\\ \\ $L_{\text{tall}}$\end{tabular}} & -5 & 50 & \multicolumn{1}{l|}{7099} & 4804 & 32.33 & \multicolumn{1}{l|}{142795.50} & 96299.73 & 32.56 \\ \cline{2-9}
 & \multirow{2}{*}{5} & 25 & \multicolumn{1}{l|}{535} & 431 & 19.44 & \multicolumn{1}{l|}{3556.64} & 2980.70 & 16.19 \\ \cline{3-9}
 &  & 50 & \multicolumn{1}{l|}{785} & 731 & 6.88 & \multicolumn{1}{l|}{15342.95} & 14582.81 & 4.95 \\ \cline{2-9}
 & -10 & 50 & \multicolumn{1}{l|}{3600} & 2066 & 42.61 & \multicolumn{1}{l|}{70585.92} & 42959.31 & 39.14 \\ \cline{2-9}
 & \multirow{2}{*}{10} & 25 & \multicolumn{1}{l|}{1514} & 1007 & 33.49 & \multicolumn{1}{l|}{12124.34} & 8114.281 & 33.07 \\ \cline{3-9}
 &  & 50 & \multicolumn{1}{l|}{1399} & 1322 & 5.50 & \multicolumn{1}{l|}{27627.00} & 27202.81 & 1.54 \\ \hline
\multirow{4}{*}{\begin{tabular}[c]{@{}l@{}}lp\_ken\_18$^{\rm T}$\\ \\ $L_{\text{tall}}$\end{tabular}} & \multirow{2}{*}{5} & 25 & \multicolumn{1}{l|}{1540} & 1526 & 0.91 & \multicolumn{1}{l|}{69137.03} & 66817.27 & 3.36 \\ \cline{3-9}
 &  & 50 & \multicolumn{1}{l|}{257} & 254 & 1.17 & \multicolumn{1}{l|}{36905.97} & 36963.28 & -0.16 \\ \cline{2-9}
 & \multirow{2}{*}{10} & 25 & \multicolumn{1}{l|}{225} & 211 & 6.22 & \multicolumn{1}{l|}{9201.80} & 8529.42 & 7.31 \\ \cline{3-9}
 &  & 50 & \multicolumn{1}{l|}{291} & 290 & 0.34 & \multicolumn{1}{l|}{38768.55} & 38806.09 & -0.10 \\ \hline
\multirow{4}{*}{\begin{tabular}[c]{@{}l@{}}appu\\ \\ $L_{\text{flat}}$\end{tabular}} & \multirow{2}{*}{-5} & 25 & \multicolumn{1}{l|}{3648} & 2459 & 32.59 & \multicolumn{1}{l|}{119809.60} & 80730.16 & 32.62 \\ \cline{3-9}
 &  & 50 & \multicolumn{1}{l|}{618} & 523 & 15.37 & \multicolumn{1}{l|}{43332.45} & 36430.81 & 15.93 \\ \cline{2-9}
 & \multirow{2}{*}{-10} & 25 & \multicolumn{1}{l|}{8523} & 6238 & 26.81 & \multicolumn{1}{l|}{316122.20} & 231805.20 & 26.67 \\ \cline{3-9}
 &  & 50 & \multicolumn{1}{l|}{947} & 650 & 31.36 & \multicolumn{1}{l|}{68717.23} & 48056.67 & 30.07 \\ \hline
\multirow{4}{*}{\begin{tabular}[c]{@{}l@{}}seymourl\\ \\ $L_{\text{flat}}$\end{tabular}} & \multirow{2}{*}{-5} & 25 & \multicolumn{1}{l|}{204} & 106 & 48.04 & \multicolumn{1}{l|}{817.53} & 418.02 & 48.87 \\ \cline{3-9}
 &  & 50 & \multicolumn{1}{l|}{45} & 39 & 13.33 & \multicolumn{1}{l|}{479.72} & 420.11 & 12.43 \\ \cline{2-9}
 & \multirow{2}{*}{-10} & 25 & \multicolumn{1}{l|}{499} & 306 & 38.68 & \multicolumn{1}{l|}{1425.69} & 870.55 & 38.94 \\ \cline{3-9}
 &  & 50 & \multicolumn{1}{l|}{70} & 63 & 10.00 & \multicolumn{1}{l|}{671.69} & 612.64 & 8.79 \\ \hline
\multirow{4}{*}{\begin{tabular}[c]{@{}l@{}}lp\_dfl001$^{\rm T}$\\ \\ $L_{\text{flat}}$\end{tabular}} & \multirow{2}{*}{-5} & 25 & \multicolumn{1}{l|}{34} & 34 & 0.00 & \multicolumn{1}{l|}{130.56} & 128.95 & 1.23 \\ \cline{3-9}
 &  & 50 & \multicolumn{1}{l|}{12} & 12 & 0.00 & \multicolumn{1}{l|}{130.23} & 132.88 & -2.03 \\ \cline{2-9}
 & \multirow{2}{*}{-10} & 25 & \multicolumn{1}{l|}{40} & 40 & 0.00 & \multicolumn{1}{l|}{114.58} & 113.16 & 1.24 \\ \cline{3-9}
 &  & 50 & \multicolumn{1}{l|}{12} & 12 & 0.00 & \multicolumn{1}{l|}{117.38} & 117.58 & -0.17 \\ \hline
\end{tabular}%
}
\end{table}

As shown in \Cref{tab: efficiency of IRJBD and IRRJBD},
IRRJBD used (considerably) fewer restarts and less CPU time than IRJBD, and the speedup ratios
were $30\%\sim 50\%$ in two-thirds of the cases.
Only for the last problem in the table, the two algorithms
consumed the same restarts and very comparable CPU time. Therefore, as a whole,
we can conclude that IRRJBD is
generally more efficient and is often much more efficient than IRJBD.

Next, we examine the ultimately attainable accuracy of IRJBD and IRRJBD with
$lsqrtol=10\epsilon_{\rm mach}$.
Table~\ref{tab: final accuracy of IRJBD and IRRJBD} reports the results.

\begin{table}[htbp!]
\caption{Final accuracy of IRJBD and IRRJBD.}
\label{tab: final accuracy of IRJBD and IRRJBD}
\resizebox{\textwidth}{!}{%
\begin{tabular}{|l|l|l|ll|ll|ll|}
\hline
\multirow{2}{*}{$(A;L)$} & \multicolumn{1}{c|}{\multirow{2}{*}{$target$}} & \multicolumn{1}{c|}{\multirow{2}{*}{$k_{\max}$}} & \multicolumn{2}{c|}{$iter$} & \multicolumn{2}{c|}{$time$} & \multicolumn{2}{c|}{$Res$} \\ \cline{4-9}
 & \multicolumn{1}{c|}{} & \multicolumn{1}{c|}{} & \multicolumn{1}{l|}{IR} & IRR & \multicolumn{1}{l|}{IR} & IRR & \multicolumn{1}{l|}{IR} & IRR \\ \hline
\multirow{4}{*}{\begin{tabular}[c]{@{}l@{}}lp\_ken\_18$^{\rm T}$\\ \\ $L_{\text{tall}}$\end{tabular}} & \multirow{2}{*}{5} & 25 & \multicolumn{1}{l|}{2763} & 2675 & \multicolumn{1}{l|}{204840.17} & 162964.53 & \multicolumn{1}{l|}{2.54E-13} & 2.49E-13 \\ \cline{3-9}
 &  & 50 & \multicolumn{1}{l|}{506} & 486 & \multicolumn{1}{l|}{79503.48} & 79460.89 & \multicolumn{1}{l|}{8.56E-13} & 8.48E-13 \\ \cline{2-9}
 & \multirow{2}{*}{10} & 25 & \multicolumn{1}{l|}{293} & 276 & \multicolumn{1}{l|}{12759.09} & 11681.25 & \multicolumn{1}{l|}{5.26E-13} & 4.40E-13 \\ \cline{3-9}
 &  & 50 & \multicolumn{1}{l|}{376} & 364 & \multicolumn{1}{l|}{51068.14} & 51990.38 & \multicolumn{1}{l|}{2.19E-12} & 2.12E-12 \\ \hline
\multirow{4}{*}{\begin{tabular}[c]{@{}l@{}}appu\\ \\ $L_{\text{flat}}$\end{tabular}} & \multirow{2}{*}{-5} & 25 & \multicolumn{1}{l|}{6592} & 4748 & \multicolumn{1}{l|}{282070.83} & 203291.63 & \multicolumn{1}{l|}{4.71E-13} & 3.48E-13 \\ \cline{3-9}
 &  & 50 & \multicolumn{1}{l|}{1107} & 860 & \multicolumn{1}{l|}{107983.42} & 87503.20 & \multicolumn{1}{l|}{1.91E-13} & 1.53E-13 \\ \cline{2-9}
 & \multirow{2}{*}{-10} & 25 & \multicolumn{1}{l|}{10000} & 8525 & \multicolumn{1}{l|}{447939.72} & 405786.81 & \multicolumn{1}{l|}{1.20E-10} & 4.52E-13 \\ \cline{3-9}
 &  & 50 & \multicolumn{1}{l|}{1709} & 1467 & \multicolumn{1}{l|}{168235.16} & 141258.48 & \multicolumn{1}{l|}{2.68E-13} & 2.26E-13 \\ \hline
\multirow{4}{*}{\begin{tabular}[c]{@{}l@{}}seymourl\\ \\ $L_{\text{flat}}$\end{tabular}} & \multirow{2}{*}{-5} & 25 & \multicolumn{1}{l|}{334} & 190 & \multicolumn{1}{l|}{2903.95} & 1653.94 & \multicolumn{1}{l|}{1.25E-14} & 9.08E-15 \\ \cline{3-9}
 &  & 50 & \multicolumn{1}{l|}{69} & 59 & \multicolumn{1}{l|}{1338.11} & 1138.39 & \multicolumn{1}{l|}{7.72E-15} & 6.73E-15 \\ \cline{2-9}
 & \multirow{2}{*}{-10} & 25 & \multicolumn{1}{l|}{802} & 560 & \multicolumn{1}{l|}{7017.11} & 4993.51 & \multicolumn{1}{l|}{2.27E-14} & 1.56E-14 \\ \cline{3-9}
 &  & 50 & \multicolumn{1}{l|}{105} & 95 & \multicolumn{1}{l|}{1905.55} & 1712.77 & \multicolumn{1}{l|}{9.07E-15} & 8.15E-15 \\ \hline
\end{tabular}%
}
\end{table}

Table~\ref{tab: final accuracy of IRJBD and IRRJBD} shows that,
except the computation of ten smallest GSVD components of
$\{{\rm appu}, L_{\rm tall}\}$ using
IRJBD with subspace dimension $k_{\max}=25$ where the algorithm failed after 10000 restarts,
the actual residual norms obtained by IRJBD and IRRJBD achieved the level of $\epsilon_{\rm mach}$,
which confirms \cref{lsaccuracy} by noticing that the sizes of
$\kappa([A;L])$'s in Table~\ref{tab:matrices} are modest.
As we see, IRRJBD required significantly fewer restarts and less CPU time
than IRJBD in at least half of the cases,
which are similar to the results in Table \ref{tab: efficiency of IRJBD and IRRJBD}
where $tol=10^{-8}$.

\begin{table}[htp!]
\caption{Effectiveness and efficiency of IRRJBD and IRJBD for flat $L$. $Res$ in bold denotes
no convergence.}
\label{tab: robustness of IRJBD and IRRJBD, tall A or flat L}
\begin{minipage}{\textwidth}
\centering
\begin{subtable}{\textwidth}
\centering
\caption{$\{A,L\}=\{\text{appu},L_{\text{flat}}\}$}
\resizebox{\textwidth}{!}{%
\begin{tabular}{|l|ll|ll|ll|ll|}
\hline
$target$ & \multicolumn{2}{c|}{-5} & \multicolumn{2}{c|}{5} & \multicolumn{2}{c|}{-10} & \multicolumn{2}{c|}{10} \\ \hline
$k_{\max}$ & \multicolumn{1}{l|}{25} & 50 & \multicolumn{1}{l|}{25} & 50 & \multicolumn{1}{l|}{25} & 50 & \multicolumn{1}{l|}{25} & 50 \\ \hline
$iter_{\text{IR}}$ & \multicolumn{1}{l|}{3648} & 618 & \multicolumn{1}{l|}{62} & 18 & \multicolumn{1}{l|}{8523} & 947 & \multicolumn{1}{l|}{541} & 73 \\ \hline
$iter_{\text{IRR}}$ & \multicolumn{1}{l|}{2459} & 523 & \multicolumn{1}{l|}{60} & 24 & \multicolumn{1}{l|}{6238} & 650 & \multicolumn{1}{l|}{746} & 74 \\ \hline
$SI$(\%) & \multicolumn{1}{l|}{32.59} & 15.37 & \multicolumn{1}{l|}{-} & - & \multicolumn{1}{l|}{26.81} & 31.36 & \multicolumn{1}{l|}{-} & - \\ \hline
$time_{\text{IR}}$ & \multicolumn{1}{l|}{119809.60} & 43332.45 & \multicolumn{1}{l|}{1982.52} & 1228.11 & \multicolumn{1}{l|}{316122.20} & 68717.23 & \multicolumn{1}{l|}{20126.77} & 5298.77 \\ \hline
$time_{\text{IRR}}$ & \multicolumn{1}{l|}{80730.16} & 36430.81 & \multicolumn{1}{l|}{1988.05} & 1630.11 & \multicolumn{1}{l|}{231805.20} & 48056.67 & \multicolumn{1}{l|}{28108.64} & 5390.69 \\ \hline
$ST$(\%) & \multicolumn{1}{l|}{32.62} & 15.93 & \multicolumn{1}{l|}{-} & - & \multicolumn{1}{l|}{26.67} & 30.07 & \multicolumn{1}{l|}{-} & - \\ \hline
$Res_{\text{IR}}$ & \multicolumn{1}{l|}{1.67E-09} & 3.41E-09 & \multicolumn{1}{l|}{{\bf 4.09E-02}} & {\bf 4.09E-02} & \multicolumn{1}{l|}{1.30E-09} & 1.54E-09 & \multicolumn{1}{l|}{{\bf 4.12E-02}} &{\bf 4.09E-02} \\ \hline
$Res_{\text{IRR}}$ & \multicolumn{1}{l|}{6.38E-10} & 3.12E-09 & \multicolumn{1}{l|}{\bf 4.09E-02} & {\bf 4.09E-02} & \multicolumn{1}{l|}{6.41E-10} & 1.48E-09 & \multicolumn{1}{l|}{{\bf 4.12E-02}} & {\bf 4.09E-02} \\ \hline
$B_{\text{IR}}$ & \multicolumn{1}{l|}{6.36E+01} & 5.41E+01 & \multicolumn{1}{l|}{5.64E+09} & 6.91E+09 & \multicolumn{1}{l|}{6.82E+01} & 9.92E+02 & \multicolumn{1}{l|}{5.28E+09} & 5.91E+09 \\ \hline
$B_{\text{IRR}}$ & \multicolumn{1}{l|}{7.48E+02} & 1.45E+02 & \multicolumn{1}{l|}{3.54E+09} & 5.68E+09 & \multicolumn{1}{l|}{5.09E+02} & 1.80E+04 & \multicolumn{1}{l|}{6.20E+09} & 6.30E+09 \\ \hline
\end{tabular}%
}
\end{subtable}
\vspace{0.8em} 
\begin{subtable}{\textwidth}
\centering
\caption{$\{A,L\}=\{\text{seymourl},L_{\text{flat}}\}$}
\resizebox{\textwidth}{!}{%
\begin{tabular}{|l|ll|ll|ll|ll|}
\hline
$target$ & \multicolumn{2}{c|}{-5} & \multicolumn{2}{c|}{5} & \multicolumn{2}{c|}{-10} & \multicolumn{2}{c|}{10} \\ \hline
$k_{\max}$ & \multicolumn{1}{l|}{25} & 50 & \multicolumn{1}{l|}{25} & 50 & \multicolumn{1}{l|}{25} & 50 & \multicolumn{1}{l|}{25} & 50 \\ \hline
$iter_{\text{IR}}$ & \multicolumn{1}{l|}{204} & 45 & \multicolumn{1}{l|}{24} & 8 & \multicolumn{1}{l|}{499} & 70 & \multicolumn{1}{l|}{57} & 17 \\ \hline
$iter_{\text{IRR}}$ & \multicolumn{1}{l|}{106} & 39 & \multicolumn{1}{l|}{20} & 90 & \multicolumn{1}{l|}{306} & 63 & \multicolumn{1}{l|}{69} & 73 \\ \hline
$SI$(\%) & \multicolumn{1}{l|}{48.04} & 13.33 & \multicolumn{1}{l|}{--} & -- & \multicolumn{1}{l|}{38.68} & 10.00 & \multicolumn{1}{l|}{--} & -- \\ \hline
$time_{\text{IR}}$ & \multicolumn{1}{l|}{817.53} & 479.72 & \multicolumn{1}{l|}{96.42} & 86.89 & \multicolumn{1}{l|}{1425.69} & 671.69 & \multicolumn{1}{l|}{165.00} & 166.69 \\ \hline
$time_{\text{IRR}}$ & \multicolumn{1}{l|}{418.02} & 420.11 & \multicolumn{1}{l|}{81.88} & 961.11 & \multicolumn{1}{l|}{870.55} & 612.64 & \multicolumn{1}{l|}{197.61} & 706.34 \\ \hline
$ST$(\%) & \multicolumn{1}{l|}{48.87} & 12.43 & \multicolumn{1}{l|}{--} & -- & \multicolumn{1}{l|}{38.94} & 8.79 & \multicolumn{1}{l|}{--} & -- \\ \hline
$Res_{\text{IR}}$ & \multicolumn{1}{l|}{4.90E-10} & 3.22E-10 & \multicolumn{1}{l|}{\bf 2.55E-02} & {\bf 2.51E-02} & \multicolumn{1}{l|}{5.51E-10} & 3.80E-10 & \multicolumn{1}{l|}{{\bf 2.98E-02}} &{\bf 2.51E-02} \\ \hline
$Res_{\text{IRR}}$ & \multicolumn{1}{l|}{4.09E-10} & 3.50E-10 & \multicolumn{1}{l|}{\bf 2.65E-02} & {\bf 2.51E-02} & \multicolumn{1}{l|}{4.43E-10} & 4.33E-10 & \multicolumn{1}{l|}{\bf 2.80E-02} & {\bf 2.49E-02} \\ \hline
$B_{\text{IR}}$ & \multicolumn{1}{l|}{6.50E+01} & 2.24E+03 & \multicolumn{1}{l|}{3.07E+08} & 3.44E+08 & \multicolumn{1}{l|}{6.17E+01} & 4.94E+01 & \multicolumn{1}{l|}{1.46E+08} & 3.72E+08 \\ \hline
$B_{\text{IRR}}$ & \multicolumn{1}{l|}{2.68E+01} & 4.46E+01 & \multicolumn{1}{l|}{3.07E+08} & 3.78E+08 & \multicolumn{1}{l|}{4.30E+01} & 4.39E+01 & \multicolumn{1}{l|}{1.47E+08} & 3.82E+08 \\ \hline
\end{tabular}%
}
\end{subtable}
\vspace{0.8em} 
\begin{subtable}{\textwidth}
\centering
\caption{$\{A,L\}=\{\text{lp\_dfl001}^{\rm T},L_{\text{flat}}\}$}
\resizebox{\textwidth}{!}{%
\begin{tabular}{|l|ll|ll|ll|ll|}
\hline
$target$ & \multicolumn{2}{c|}{-5} & \multicolumn{2}{c|}{5} & \multicolumn{2}{c|}{-10} & \multicolumn{2}{c|}{10} \\ \hline
$k_{\max}$ & \multicolumn{1}{l|}{25} & 50 & \multicolumn{1}{l|}{25} & 50 & \multicolumn{1}{l|}{25} & 50 & \multicolumn{1}{l|}{25} & 50 \\ \hline
$iter_{\text{IR}}$ & \multicolumn{1}{l|}{34} & 12 & \multicolumn{1}{l|}{14} & 5 & \multicolumn{1}{l|}{40} & 12 & \multicolumn{1}{l|}{37} & 9 \\ \hline
$iter_{\text{IRR}}$ & \multicolumn{1}{l|}{34} & 12 & \multicolumn{1}{l|}{13} & 13 & \multicolumn{1}{l|}{40} & 12 & \multicolumn{1}{l|}{50} & 13 \\ \hline
$SI$(\%) & \multicolumn{1}{l|}{0.00} & 0.00 & \multicolumn{1}{l|}{--} & -- & \multicolumn{1}{l|}{0.00} & 0.00 & \multicolumn{1}{l|}{--} & -- \\ \hline
$time_{\text{IR}}$ & \multicolumn{1}{l|}{130.56} & 130.23 & \multicolumn{1}{l|}{53.78} & 54.23 & \multicolumn{1}{l|}{114.58} & 117.38 & \multicolumn{1}{l|}{97.39} & 90.92 \\ \hline
$time_{\text{IRR}}$ & \multicolumn{1}{l|}{128.95} & 132.88 & \multicolumn{1}{l|}{48.16} & 146.14 & \multicolumn{1}{l|}{113.16} & 117.58 & \multicolumn{1}{l|}{134.17} & 129.23 \\ \hline
$ST$(\%) & \multicolumn{1}{l|}{1.23} & -2.03 & \multicolumn{1}{l|}{--} & -- & \multicolumn{1}{l|}{1.24} & -0.17 & \multicolumn{1}{l|}{--} & -- \\ \hline
$Res_{\text{IR}}$ & \multicolumn{1}{l|}{3.20E-08} & 3.21E-08 & \multicolumn{1}{l|}{\bf 2.87E-02} & {\bf 2.41E-02} & \multicolumn{1}{l|}{3.20E-08} & 3.21E-08 & \multicolumn{1}{l|}{\bf 3.88E-02} & {\bf 2.44E-02} \\ \hline
$Res_{\text{IRR}}$ & \multicolumn{1}{l|}{3.19E-08} & 3.21E-08 & \multicolumn{1}{l|}{\bf 2.91E-02} & {\bf 2.43E-02} & \multicolumn{1}{l|}{3.21E-08} & 3.21E-08 & \multicolumn{1}{l|}{\bf 3.88E-02} & {\bf 2.42E-02} \\ \hline
$B_{\text{IR}}$ & \multicolumn{1}{l|}{1.03E+04} & 4.45E+04 & \multicolumn{1}{l|}{2.56E+09} & 8.12E+09 & \multicolumn{1}{l|}{5.63E+03} & 3.35E+05 & \multicolumn{1}{l|}{8.25E+09} & 6.82E+09 \\ \hline
$B_{\text{IRR}}$ & \multicolumn{1}{l|}{1.44E+03} & 3.44E+04 & \multicolumn{1}{l|}{2.47E+09} & 7.85E+09 & \multicolumn{1}{l|}{1.53E+04} & 5.59E+04 & \multicolumn{1}{l|}{1.59E+09} & 6.66E+09 \\ \hline
\end{tabular}%
}
\end{subtable}
\end{minipage}
\end{table}

Next, we present the results for some tall $A$ or flat $L$. As has been addressed,
the JBD and RJBD methods may encounter serious difficulties for these problems because
 $\|\underline{B}_k^{-1}\|\|\widehat B_k^{-1}\|$ could be arbitrarily large.
The results are reported in Table~\ref{tab: robustness of IRJBD and IRRJBD, tall A or flat L},
where the four parts correspond to the
tall $A$ and tall $L$, square $A$ and flat $L$, flat $A$ and flat $L$, and tall $A$ and flat $L$,
respectively. For
the examples that both algorithms failed to converge, we do not compare the efficiency of the two
algorithms.

As we can see from Table~\ref{tab: robustness of IRJBD and IRRJBD, tall A or flat L}, for tall matrices
$A$, both algorithms computed the smallest GSVD components but failed to compute the largest GSVD components because $\|\underline{B}_k^{-1}\|\|\widehat B_k^{-1}\|$'s
are very large so that the $Res$ were
bigger than $tol$ by several orders. A rigorous explanation is given in \cite{IRJBD} on these
subtle differences of IRJBD, and it is directly applicable to IRRJBD, as is restated below.

From \eqref{start vectors after restart}, for the computation of the largest GSVD components,
it is known that the updated starting vector $u_1^+$ tends to a linear combination
of the eigenvectors of $Q_AQ_A^{\rm T}$ corresponding to the desired
largest eigenvalues as restarts proceed, and thus
ultimately contains {\em little} information on the eigenvectors of $Q_AQ_A^{\rm T}$
associated with its zero eigenvalues.
Since $U_k^{\rm T}Q_AQ_A^{\rm T}U_k=\underline{B}_k^{\rm T}\underline{B}_k$, it is more likely that
the smallest Ritz value, i.e., the smallest
eigenvalue of $\underline{B}_k^{\rm T}\underline{B}_k$
is not small, i.e., $\|\underline{B}_k^{-1}\|$ is not large.  On the contrary,
for the smallest GSVD components, the updated starting vector
$u_1^+$ tends to a linear combination
of the eigenvectors of $Q_AQ_A^{\rm T}$ corresponding to the zero and desired
smallest eigenvalues as restarts proceed, and thus
contains {\em non-negligible} information on the eigenvectors of $Q_AQ_A^{\rm T}$
associated with its zero eigenvalues, causing that the smallest
eigenvalue of $\underline{B}_k^{\rm T}\underline{B}_k$ can be arbitrarily small,
i.e., $\|\underline{B}_k^{-1}\|$ is large.

From Tables~\ref{tab: robustness of IRJBD and IRRJBD, square of flat A, tall L}--\ref{tab: robustness of IRJBD and IRRJBD, tall A or flat L}, we can see
that, for many of the cases, IRRJBD used substantially
fewer restarts and CPU time than IRJBD did.

Finally, we select two examples from Table \ref{tab: final accuracy of IRJBD and IRRJBD} and plot the curves of the residual norms versus the restarts for the two algorithms, as shown in Figure~\ref{fig: IRJBD and IRRJBD, nopoly,Ltall,10,25,seymourl,Lflat,-10,25}. Specifically, for each point in the figure, the value on the vertical axis denotes
the maximum residual norm among the $l$ approximate GSVD components.

\begin{figure}[htp!]
\centering
\includegraphics[width=0.9\linewidth]{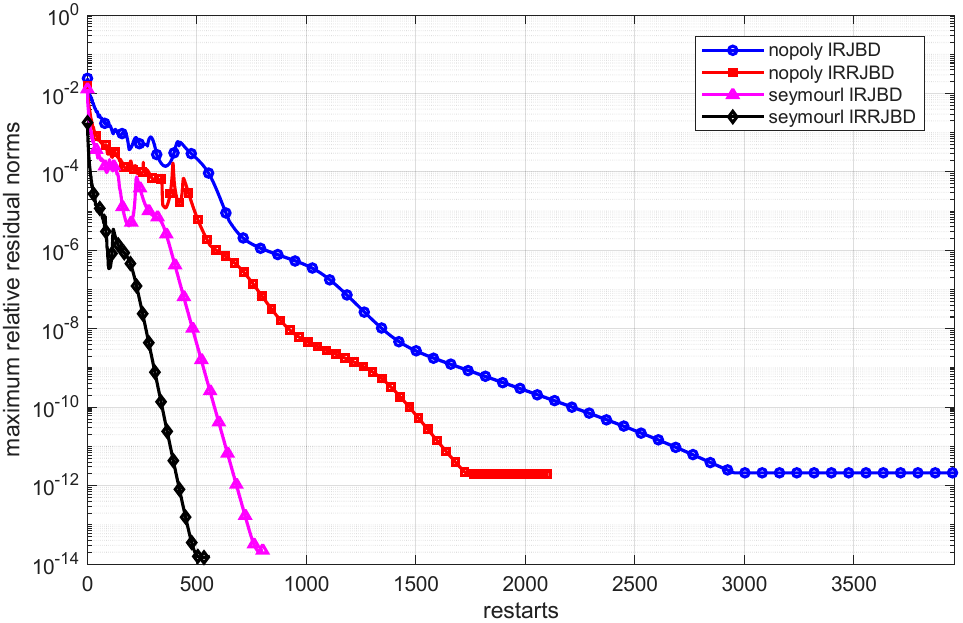}
\caption{Convergence curves of IRJBD and IRRJBD for computing the largest ten GSVD components of $\{\text{nopoly},L_{\text{tall}}\}$ and the smallest ten GSVD components of $\{\text{seymourl},L_{\text{flat}}\}$ with $k_{\max}=25$.}
\label{fig: IRJBD and IRRJBD, nopoly,Ltall,10,25,seymourl,Lflat,-10,25}
\end{figure}

We observe from the figure that the convergence of
all scenarios exhibited a little irregular behavior in the initial stage but soon it became smooth.
Furthermore, for the two problems, when the two algorithms stabilized,
IRRJBD was approximately
one and half times and nearly twice as fast as IRJBD in terms of restarts,
respectively, substantial improvements!


In summary, based on the whole experiments, we conclude that IRRJBD and IRJBD are equally
reliable, but IRRJBD is generally significantly more efficient than IRJBD.

\section{Conclusion}\label{sec:8}

We have proposed the RJBD method, which, unlike the JBD method, computes
the refined Ritz approximations that satisfy the residual optimality.
The RJBD method is suitable for computing several extreme GSVD components of a large regular matrix pair.
We have analyzed the convergence of the JBD and RJBD methods, and proved that RJBD has better convergence
than JBD and overcomes the possibly irregular convergence or non-convergence
of Ritz vectors. These results hold for general Rayleigh--Ritz projection method and refined
Rayleigh--Ritz projection method for the GSVD problem. In the meantime,
we have established the interlacing property of Ritz values and generalized singular values
as well as a general interlacing property of the matrix pair and its submatrices consisting
of its arbitrary columns,
which, to our best knowledge, are new.
We have adapted the implicit restarting technique to RJBD, and have developed an IRRJBD algorithm
with the refined shifts proposed, which are shown to be
more accurate than the exact shifts used within IRJBD. The
numerical experiments have illustrated
that IRRJBD algorithm often outperforms IRJBD significantly.

There remains an important problem to be solved for the JBD and RJBD methods. As we have
addressed in \cite{IRJBD} and this paper, in finite precision arithmetic,
JBD and RJBD may encounter severe difficulties when there are zero or infinite
generalized singular values because
$\|\underline{B}_k^{-1}\|\|\widehat B_k^{-1}\|$ can be very large. How to propose effective and efficient purifications and introduce them into the methods to ensure that
$\|\underline{B}_k^{-1}\|\|\widehat B_k^{-1}\|$ is uniformly bounded
is extremely significant and highly challenging. These will constitute our forthcoming
work. We should remind that purifications are needed to
avoid misconvergence, spurious approximations and
convergence delay in numerous projection
methods for not only generalized eigenvalue problems and GSVD problems but also
the SVD problem of a rectangular or flat matrix where the singular triplets with the
singular values closest to the target
point {\em zero} are of interest, i.e., the smallest singular triplets, which is the case in
JD-type SVD methods \cite{Hochstenbach2004BIT,HuangandJia2019JD}.

\section*{Declarations}
The two authors declare that they have no
financial interests, and they read and approved the final manuscript.
\smallskip

{\bf Data Availability.} \ Not applicable.



\end{document}